\documentclass[a4paper,12pt,reqno]{amsart}
\usepackage{amsfonts,amssymb,hyperref,amsthm,enumerate,
color,stmaryrd,pgf,tikz,comment,multirow}
\usepackage[left=2.6cm,right=2.6cm,top=3cm,bottom=3cm,bindingoffset=0cm]
{geometry}
\newtheorem{thm}{Theorem}[section]
\newtheorem{lem}[thm]{Lemma}
\newtheorem{prop}[thm]{Proposition}
\newtheorem{hypo}[thm]{Hypothesis}
\newtheorem{conj}[thm]{Conjecture}

\newtheorem*{claim}{Claim}
\theoremstyle{definition}
\newtheorem{defi}[thm]{Definition}
\def\B{\mathcal{B}}
\def\C{\mathbf{C}}
\def\G{\Gamma}
\def\K{\mathbf{K}}
\def\Z{\mathbb{Z}}
\DeclareMathOperator{\aut}{Aut}
\DeclareMathOperator{\cdc}{B}
\DeclareMathOperator{\vertex}{V}
\DeclareMathOperator{\edge}{E}

\DeclareMathOperator{\orb}{Orb}
\newcommand{\sg}[1]{\langle {#1}\rangle}
\title[Stability of Rose Window graphs]
{Stability of Rose Window graphs}
\author[M. Ahanjideh, I. Kov\'acs, K. Kutnar]
{Milad Ahanjideh$^{\, 1}$, Istv\'an Kov\'acs$^{\, 2\, *}$, Klavdija Kutnar$^{\, 3}$} 
\address{M. Ahanjideh, I.~Kov\'acs, K.~Kutnar  
\newline\indent
UP IAM, University of Primorska, Muzejski trg 2, SI-6000 Koper, Slovenia 
\newline\indent
UP FAMNIT, University of Primorska, Glagol\v jaska ulica 8, SI-6000 Koper, Slovenia}
\email{milad.ahanjideh@upr.si}
\email{istvan.kovacs@upr.si} 
\email{klavdija.kutnar@upr.si} 
\thanks{$^1$~Supported by the Slovenian Research Agency (reserach project 
N1-0208). 
\newline\indent 
$^2$~Supported by the Slovenian Research Agency 
(research program P1-0285 and research projects J1-1695, N1-0140, J1-2451, N1-0208,  J1-3001).
\newline\indent 
$^3$~Supported by the Slovenian Research Agency 
(research program P1-0285 and research projects J1-1695, J1-1715, N1-0140, J1-2451, J1-2481, N1-0209, J1-3001).
\newline \indent * Corresponding author}
\begin{document}
\maketitle

\begin{abstract}
A graph $\G$ is said to be stable if for the direct product $\G \times \K_2$, 
$\aut(\G \times \K_2)$ is isomorphic to $\aut(\G) \times  \Z_2$; otherwise, 
it is called unstable. An unstable graph is called non-trivially unstable when it 
is not bipartite and no two vertices have the same neighborhood. 
Wilson described nine families of unstable Rose Window graphs and conjectured that these contain all non-trivially unstable Rose Window graphs (2008). In this paper we show that the conjecture is true.  
\end{abstract}
\section{Introduction}\label{sec:intro}
Let $\G$ be a finite and simple graph. The \emph{direct product} 
$\G \times \K_2$ of $\G$ with the complete graph $\K_2$, also known as the \emph{canonical double cover} of $\G$, has vertex set $\vertex(\G) \times \{0,1\}$ and edges  
$\{(u,0),(v,1)\}$, where $\{u,v\} \in \edge(\G)$. Following~\cite{W2}, here we shall rather use the notation $\cdc(\G)$ for $\G \times \K_2$.  
The graph $\cdc(\G)$ admits some natural automorphisms, namely the permutation 
\begin{equation}\label{eq:expect1}
(v,i) \mapsto (v,1-i),~v \in \vertex(\G),~i=0,1;
\end{equation} 
and for every $\alpha \in \aut(\G)$, the permutation 
\begin{equation}\label{eq:expect2}
(v,i) \mapsto (\alpha(v),i),~v \in \vertex(\G),~i=0,1.
\end{equation}
The above permutations can be easily seen to form a group, which is isomorphic to 
$\aut(\G) \times \Z_2$. The graph $\G$ is said to be \emph{stable} if 
$\aut(\cdc(\G)) \cong \aut(\G) \times \Z_2$, and \emph{unstable} otherwise. 
This concept of stability was introduced by Maru\v{s}i\v{c} et al.\ ~\cite{MSZ}. It    
plays a key role in the investigation of the automorphisms of the 
direct products of arbitrary graphs (see~\cite{W-M21}) and also in the study of regular embeddings of canonical double covers on orientable surfaces (see~\cite{NS}). 
Arc-transitive graphs in terms of their stability were investigated by Surowski~\cite{S}. Recently, this topic has been studied by several authors~\cite{FH, HM, HMW21a, HMW21b, QXZ19, QXZ21}.   

It is well known that $\G$ is unstable whenever it is disconnected, or  
bipartite having a non-trivial automorphism group, or it has two vertices with 
the same neighborhood (see, e.g.,~\cite{W2}), and in each of these 
cases $\G$ is also called \emph{trivially unstable}. Necessary and sufficient   
conditions for a graph to be non-trivially unstable were given by Wilson~\cite{W2}, and 
in the same paper these conditions were also applied to 
four infinite classes of graphs: 
\begin{itemize}
\item the circulant graphs,
\item the generalized Petersen graphs,
\item the Rose Window graphs,
\item the toroidal graphs.  
\end{itemize}
Various families of unstable graphs were constructed for each class and 
Wilson conjectured that these families cover all non-trivially unstable graphs in the given class. The conjecture has been confirmed for the generalized Petersen graphs~\cite{QXZ21} and for the toroidal graphs~\cite{W-M23},  whereas it has been disproved for the class of circulant graphs~\cite{QXZ19}.  To the best of our knowledge, a characterization of the unstable circulant graphs of order $n$ and valency $k$ is known only in special cases:  $n$ is odd~\cite{FH,QXZ19}, $n=2p$ for a prime $p$~\cite{HMW21a} 
and $k \le 7$~\cite{HMW21b}. Up to now, no progress has been made on the 
Rose Window graphs. The aim of this paper is to fill in this gap and 
show that the conjecture for the Rose Window graphs is also true. 

\begin{defi}[\cite{W1}]
\label{def:W}
Let $n$ be an integer such that $n \ge 3$ and let $a, r \in \Z_n$, $a, r \ne 0$. 
The \emph{Rose Window graph} $R_n(a,r)$ is defined to have the vertex set 
$\{u_i, v_i : i \in \Z_n\}$, and edges
$$
\{u_i,u_{i+1}\},~\{v_i,v_{i+r}\},~\{ u_i,v_i\},~\{u_{i+a},v_i\},~i \in \Z_n,  
$$
where the additions in the subscripts are done in $\Z_n$. 
\end{defi}

Here and in what follows, $\Z_n$ denotes the ring of residue classes of integers 
modulo $n$. Throughout the paper we regard elements of $\Z_n$ simultaneously as integers. 
The Rose Window graphs can be regarded as tetravalent analogues of the well-studied generalized Petersen graphs. Definition~\ref{def:W} is due to Wilson \cite{W1},  who was  interested primarily in constructing edge-transitive Rose Window graphs and rotary maps with underlying Rose Window graphs. It was shown later that all edge-transitive Rose Window graphs are covered by Wilson's constructions~\cite{KKM}. For various  interesting properties of Rose Window graphs, we refer to~\cite{DKM, HRS, JMM, KKM}. 

\renewcommand{\arraystretch}{1.25}
\begin{table}[t!]
{\small 
\begin{tabular}{|c|c|c|} \hline
family & $(n,a,r)$ & conditions \\ \hline
W1 & $(4m,2m,r)$ & -- \\ \hline
W2 & $(2m,m+2,m+1)$ & -- \\ \hline
W3 & $(m,2d,1)$ & $m$ is odd, $m \ge 5$, $d^2 \equiv \pm 1\!\!\pmod m$, 
$d \not\equiv \pm 1\!\!\pmod m$. \\ \hline
W4 & $(2m,a,m-1)$ & $a \ne m$ \\ \hline 
W5 & $(8m,a,2m)$ & $a$ is even \\ \hline 
W6 & $(4m,m,r)$, & $m, r$ are odd, $m \ge 3$ \\ \hline 
W7 & $(2m,m,m-1)$ & $m$ is odd \\ \hline
W8 & $(2m,m,r)$ & $m$ is odd, $r$ is even,   
$r^2 \equiv \pm 1\!\! \pmod m$, $r \not\equiv \pm 1 \!\!\pmod m$ \\ \hline
W9 & $(2m,a,m-s)$ & $s^2 \equiv 1 \!\!\pmod n$, $sa \equiv -a \!\!\pmod n$, $a < m$. \\ \hline 
\end{tabular}
}
\\ [+1.5ex]
\caption{The nine families of unstable Rose Window graphs $R_n(a,r)$.}
\end{table}

Regarding stability, Wilson found nine families of unstable Rose Window graphs, which are listed in Table~1 (see \cite[Theorems~R.1--R.9]{W2}).  
He checked by computer that these families contain all non-trivially unstable Rose Window graphs on up to $200$ vertices and posed the following conjecture. 

\begin{conj}[\cite{W2}]
Every non-trivially unstable Rose Window graph is isomorphic to a graph in one of the families W1--W9 given in Table~1.
\end{conj}

In this paper we show that the conjecture is true by proving the following theorem.

\begin{thm}\label{main}
Let $\G=R_n(a,r)$ be a non-trivially unstable Rose Window graph. 
\begin{enumerate}[{\rm (i)}]
\item If $n$ is odd, then $\G$ is isomorphic to a graph in the family W3. 
\item If $n$ is even and $\cdc(\G)$ is not edge-transitive, then 
$\G$ is isomorphic to a graph in one of the families W1 and W4--W9.
\item If $n$ is even and $\cdc(\G)$ is edge-transitive, then $\G$ is isomorphic to 
a graph in one of the families W2 and W4. 
\end{enumerate}
\end{thm}  
\section{Preliminaries}\label{sec:known}
All groups in this paper will be finite and all graphs will be finite and simple.  
If $\G$ is a graph, then $\vertex(\G)$, $\edge(\G)$ and $\aut(\G)$ denote its vertex set, edge set and automorphism group, respectively. 
Given two vertices $u, v \in \vertex(\G)$, we write $u \sim v$ when $u$ and $v$ are adjacent. For $v \in \vertex(\G)$, the set of vertices adjacent with $v$ is denoted by $\G(v)$, and for a non-empty subset $S \subseteq \vertex(\G)$, the subgraph of $\G$ induced by $S$ is denoted by $\G[S]$. Suppose that $\G$ is bipartite, i.e., $\vertex(\G)$ can be partitioned into two sets 
$V_1$ and $V_2$ such that no edge of $\G$ lies entirely in either $V_1$ or $V_2$. Throughout the paper we will refer to the sets $V_1$ and $V_2$ as biparts of $\G$. 
By $\K_n$ and $\C_n$ we denote the complete graph and the cycle on $n$ vertices, respectively, and the symbol $n\G$ will be used for $n$ disjoint copies of $\G$. 

If $G \le \aut(\G)$, then the stabilizer of $v$ in $G$ is denoted by 
$G_v$, the orbit of $v$ under $G$ by $\orb_G(v)$, and the set of all orbits under $G$ by $\orb(G,\vertex(\G))$. 
If $S \subseteq \vertex(\G)$, then the setwise and the pointwise stabilizers of $S$ in $G$ are denoted by $G_{\{S\}}$ and $G_{(S)}$, respectively. 

Let $s \ge 1$ be an integer. An \emph{$s$-arc} of $\G$ is an $(s+1)$-tuple 
$(v_1,\ldots,v_{s+1})$ of its vertices such that $v_i$ is adjacent with 
$v_{i+1}$ for each $1 \le i \le s$, and $v_i \ne v_{i+2}$ for each $1 \le i \le s-1$. 
The graph $\G$ is said to be \emph{$(G,s)$-arc-transitive} if $G \le \aut(\G)$ and $G$ is transitive on the set of all $s$-arcs of $\G$. A $(G,s)$-arc-transitive graph is called  
\emph{$(G,s)$-transitive} if it is not $(G,s+1)$-arc-transitive.    
An $(\aut(\G),s)$-transitive graph is also 
called \emph{$s$-transitive}, and an
$(\aut(\G),1)$-arc-transitive graph is simply called \emph{arc-transitive}.
\medskip

We shall need the following result about $(G,s)$-transitive tetravalent graphs. 

\begin{prop}[{\rm \cite[Chapter~17]{Bbook}}]\label{s-AT}
Let $\G$ be a connected $(G,s)$-transitive tetravalent graph with $s \le 2$ 
and let $v \in \vertex(\G)$. Then exactly one of following statements is true.
\begin{enumerate}[{\rm (i)}]
\item $s=1$ and $G_v$ is a $2$-group.
\item $s=2$ and $G_v \cong A_4$ or $S_4$. 
\end{enumerate}
\end{prop}

Let $G \le \aut(\cdc(\G))$ be the group generated by the automorphisms defined in \eqref{eq:expect1} and \eqref{eq:expect2}. 
The elements in $G$ are referred to as the \emph{expected automorphisms} of 
$\cdc(\G)$.
  
\begin{prop}[{\rm \cite{HM}}]\label{HM}
Let $\G$ be a connected and non-bipartite graph. An automorphism  
$\sigma$ of $\cdc(\G)$ is expected 
if and only if for every $v \in \vertex(\G)$, there exists some $w \in \vertex(\G)$ such that  
$$
\sigma(\{(v, 0),(v, 1)\})=\{(w,0),(w,1)\}.
$$
\end{prop}

We review next  four results about Rose Window graphs obtained in 
\cite{DKM,KKM,W1}. The first one establishes some obvious isomorphisms.

\begin{prop}[{\rm \cite{W1}}]\label{iso}
A Rose Window graph $R_n(a,r)$ is isomorphic to each of the graphs: 
$R_n(-a,r), R_n(a,-r)$ and $R_n(-a,-r)$.
\end{prop}

As a corollary of Proposition~\ref{iso}, up to graph isomorphism, one may always assume that $1 \le a, r \le n/2$ for every graph $R_n(a,r)$. 
\medskip

It was pointed out in \cite{W1} that any Rose Window graph $\G$ admits 
at most $3$ edge-orbits under $\aut(\G)$.  The next two results provide information about the order of $\aut(\G)$ when the number of edge-orbits is $3$ or $2$. 

\begin{prop}[{\rm \cite[Corollary~3.5(a)]{DKM}}]\label{DKM1}
Suppose that $\G=R_n(a,r)$ admits three edge-orbits under $\aut(\G)$. 
If $2a \not\equiv 0\!\!\pmod n$, then $|\aut(\G)|=2n$.
\end{prop} 

\begin{prop}[{\rm \cite[Corollary~3.9]{DKM}}]\label{DKM2}
Suppose that $\G=R_n(a,r)$ admits two edge-orbits under $\aut(\G)$. 
If $2a \not\equiv 0\!\!\pmod n$, then $|\aut(\G)|=4n$.
\end{prop} 

\begin{prop}[{\rm \cite[Lemma~2.1 \& Theorem~2.3]{DKM}}]\label{DKM3}
Suppose that $\G=R_n(a,r)$ is not edge-transitive. Then there are 
two edge-orbits under $\aut(\G)$ if and only if 
$$
r^2 \equiv 1\!\!\!\!\!\pmod n~\text{and}~ra \equiv  \pm a\!\!\!\!\!\pmod n.
$$ 
\end{prop} 

The following four families of edge-transitive graphs were introduced by Wilson~\cite{W1}.

\begin{prop}[{\rm \cite[Theorem~1.2]{KKM}}]\label{KKM1}
Let $\G=R_n(a,r)$ be an edge-transitive graph and suppose that 
$1 \le a, r \le n/2$. Then one of the following holds. 
\begin{enumerate}[{\rm (a)}]
\item $\G=R_n(2,1)$,
\item $\G=R_{2m}(m-2,m-1)$,
\item $\G=R_{12m}(3m+2,3m-1)$ or $R_{12m}(3m-2,3m+1)$, 
\item $\G=R_{2m}(2b,r)$, $b^2 \equiv \pm 1\!\!\pmod m$, $r \in \{1,m-1\}$ and $r$ is odd. 
\end{enumerate}
Moreover, in family (d), $H \lhd \aut(\G)$, where $H$ is generated by 
the permutation 
$$
(u_0,u_2,\ldots,u_{2m-2})(u_1,u_3,\ldots,u_{2m-1})
(v_0,v_2,\ldots,v_{2m-2})(v_1,v_3,\ldots,v_{2m-1}).
$$
\end{prop} 

A subgroup $H$ of a group $G$ is \emph{core-free} if it contains no 
non-trivial normal subgroup of $G$. 
The following result of Lucchini~\cite{L},  independently discovered by Herzog and Kaplan \cite{HK}, about cyclic core-free subgroups in 
finite groups will be needed in the proof of the main result of this paper.

\begin{prop}[{\rm \cite{HK,L}}]\label{L}
If $C$ is a core-free cyclic proper subgroup of a group $G$, then 
$|C|^2 < |G|$.
\end{prop}
\section{On the canonical double covers of Rose Window graphs}
\label{sec:cdcRW}
In this section we collect a couple of elementary properties of the 
graphs in the title. Throughout the section we let $\G=R_n(a,r)$. 
\medskip

The canonical double cover $\cdc(\G)$ has vertex set  
$$
\{u_{i,j}, v_{i,j} : i \in \Z_n, j \in \Z_2\}, 
$$
and edges
$$
\{u_{i,j},u_{i+1,j+1}\},~\{v_{i,j},v_{i+r,j+1}\},~\{u_{i,j},v_{i,j+1}\},~\{u_{i+a,j},v_{i,j+1}\},
$$
where $i \in \Z_n, j \in \Z_2$ and the additions in the subscripts are done in $\Z_n$ and $\Z_2$, respectively. 
\subsection{The rim, the hub and the spoke cycles}
The subgraph of $\cdc(\G)$ induced by the edges 
$\{u_{i,j},u_{i+1,j+1}\}$ ($i \in \Z_n, j \in \Z_2$) is a cycle of length $2n$ if $n$ is odd, and two cycles of the same length $n$ if $n$ is even. These cycles will be referred to as 
the \emph{ rim cycles} of $\cdc(\G)$. 

If $2r \not\equiv 0\!\!\pmod n$, then the subgraph induced by the edges 
$\{v_{i,j},v_{i+r,j+1}\}$ ($i \in \Z_n, j \in \Z_2$) is a union of disjoint cycles of the same 
length and these will be referred to as the {\em hub cycles}. 

Finally, the subgraph induced by the remaining edges 
$\{u_{i,j},v_{i,j+1}\}, \{u_{i+a,j},v_{i,j+1}\}$ ($i \in \Z_n, j \in \Z_2$) is a 
union of cycles of the same length and these will be referred to as the 
{\em spoke cycles}. The next lemma is straightforward.

\begin{lem}
Let $\G=R_n(a,r)$ such that 
$2r \not\equiv 0\!\!\pmod n$, and let $\ell_h$ and $\ell_s$ be the length of the hub 
and the spoke cycles of $\cdc(\G)$, respectively. Then 
\begin{equation}\label{eq:ell}
\ell_h=\frac{2^\varepsilon n}{\gcd(r,n)}\quad \text{and} \quad
\ell_s=\frac{2 n}{\gcd(a,n)},
\end{equation}
where $\varepsilon=0$ if $n/\gcd(r,n)$ is even and 
$\varepsilon=1$ otherwise. 
\end{lem}
\subsection{The partition $\{S_i : 1 \le i \le 4\}$}
Suppose that $n$ is even. The following sets of vertices are well-defined:
\begin{equation}\label{eq:S12}
S_1=\{u_{i,j} : i \equiv j \!\!\!\!\pmod 2\},~
S_2=\{ u_{i,j} : i \not\equiv j \!\!\!\!\pmod 2\}, 
\end{equation}
\begin{equation}\label{eq:S34}
S_3=\{v_{i,j} : i \equiv j \!\!\!\!\pmod 2\},~ 
S_4=\{ v_{i,j} : i \not\equiv j \!\!\!\!\pmod 2\}.
\end{equation}
The above sets form a partition of $\vertex(\cdc(\G))$. 
Note that the vertices visited by the two rim cycles of $\cdc(\G)$ form exactly the sets 
$S_1$ and $S_2$, respectively 
\subsection{Connectivity}
\begin{lem}\label{bipartite}
The following conditions are equivalent for any Rose Window graph $\G=R_n(a,r)$.
\begin{enumerate}[{\rm (i)}]
\item $\G$ is bipartite. 
\item $\cdc(\G)$ is disconnected.
\item $n, a$ are even and $r$ is odd.
\end{enumerate}
\end{lem}
\begin{proof}  
(i) $\Rightarrow$ (ii): This follows from the known fact that the canonical 
double cover of any bipartite graph is disconnected (see, e.g. \cite{W2}). 

(ii) $\Rightarrow$ (iii): The number $n$ must be even, for otherwise there 
is only one rim cycle and $\cdc(\G)$ is clearly connected. 
If $a$ were odd, then one can easily check that every spoke cycle would contain vertices 
from both rim cycles, so $\cdc(\G)$ would be connected. This shows that 
$a$ is even. Note that, if $X$ is a spoke cycle, then $\vertex(X) \subseteq S_1 \cup S_4$ 
or $\vertex(X) \subseteq S_2 \cup S_3$. 

Finally, if $r$ were even, then every hub cycle would contain vertices from both sets 
$S_3$ and $S_4$, and so $\cdc(\G)$ would be connected. This shows that 
$r$ is odd.  

(iii) $\Rightarrow$ (i): It is straightforward to check that $\G$ is bipartite with 
biparts $V_1$ and $V_2$ defined as 
$V_1=\{u_i, v_{i+1} : i \equiv 0 \!\!\pmod 2\}$ and 
$V_2=\{u_i, v_{i+1} : i \equiv 1 \!\!\pmod 2\}$.
\end{proof}
\subsection{Some expected automorphisms}
Define the permutations $\rho, \mu$ and $\beta$ of $\vertex(\cdc(\G))$ 
as follows: 
\begin{eqnarray}
\rho &=& \prod_{j \in \Z_2}(u_{0,j},u_{1,j},\ldots,u_{n-1,j})
(v_{0,j},v_{1,j},\ldots,v_{n-1,j}). \label{eq:rho} \\
\mu &=& \prod_{i \in \Z_n, j \in \Z_2}(u_{i,j}, u_{-i,j})(v_{i,j}, 
v_{-i-a,j}) \label{eq:mu} \\ 
\beta &=& \prod_{i\in \Z_n}(u_{i,0},u_{i,1})(v_{i,0},v_{i,1}). \label{eq:beta}
\end{eqnarray}
It is easy to verify that all these permutations are 
expected automorphism of $\cdc(\G)$. Also, 
$\sg{\rho,\mu} \cong D_{2n}$ and $\beta$ commutes with both $\rho$ and 
$\mu$. Here $D_{2n}$ denotes the dihedral group of order $2n$. 
\medskip

The next result can be obtained directly from Proposition~\ref{HM}.

\begin{lem}\label{cor-HM}
Let $\G=R_n(a,r)$ be a connected and non-bipartite graph and let $\beta$ be 
the automorphism of $\cdc(\G)$ defined in \eqref{eq:beta}. An automorphism  
$\sigma$ of $\cdc(\G)$ is expected 
if and only if $\sigma \beta=\beta \sigma$.
\end{lem}

\begin{lem}\label{pointwise}
Let $\G=R_n(a,r)$ be a graph such that $n$ is even and 
one of the following conditions hold.
\begin{enumerate}[{\rm (i)}]
\item $a, r$ are even, $2a \not\equiv 0\!\! \pmod n$ and 
$4r \not\equiv 0\!\! \pmod n$.  
\item $a$ is odd and $4a \not\equiv 0\!\! \pmod n$.
\item $a$ is odd, $4a \equiv 0\!\! \pmod n$, $2a \not\equiv 0\!\! \pmod n$ 
and $r$ is even. 
\end{enumerate}
Then the only automorphism of $\cdc(\G)$ fixing 
$S_1$ pointwise and $S_2$ setwise is the identity permutation, where 
$S_1$ and $S_2$ are the sets defined in \eqref{eq:S12}. 
\end{lem}
\begin{proof}
Let $\sigma$ be an automorphism of $\cdc(\G)$ fixing $S_1$ pointwise and 
$S_2$ setwise. Note that $\sigma$ maps any rim/hub/spoke cycle to a rim/hub/spoke cycle.

(i): Let $X$ be any spoke cycle having a vertex 
from $S_1$. Since $a$ is even, it follows that 
$|\vertex(X) \cap S_1|=|\vertex(X) \cap S_4|=\ell_s/2$, where $S_4$ is the set 
defined in \eqref{eq:S34}. 
As $2a \not\equiv 0\!\! \pmod n$, $\ell_s > 4$ due to \eqref{eq:ell}, we 
see that $\sigma$ fixes every vertex in $S_4$ too. Now, as $r$ is even, for every hub cycle $X$, it holds that $|\vertex(X) \cap S_3|=|\vertex(X) \cap S_4|=\ell_h/2$, where $S_3$ is the set 
defined in \eqref{eq:S34}.   
As $4r \not\equiv 0\!\! \pmod n$, $\ell_h > 4$ due to \eqref{eq:ell}, and this yields that 
$\sigma$ fixes every vertex in $S_3$ too. It is easy to see that 
$\sigma$ is the identity permutation. 

(ii): In this case, $|X \cap S_1|=\ell_s/4$ for every spoke cycle $X$. 
Since $a$ is odd and $4a \not\equiv 0\!\! \pmod n$, it follows from \eqref{eq:ell} that $\ell_s > 8$, which shows that $\sigma$ is the identity permutation. 

(iii): Fix an arbitrary vertex $u_{i,j} \in S_2$ and let $X$ be the spoke cycle through 
$u_{i,j}$. Then 
$$
X=\big(\, u_{i,j},v_{i,j+1},u_{i+a,j},v_{i+a,j+1},u_{i+2a,j},v_{i+2a,j+1},u_{i+3a,j},
v_{i+3a,j+1}\, \big).
$$
Then $u_{i+a,j}, u_{i+3a,j} \in S_1$, hence $\sigma$ maps $X$ to itself, 
and either it fixes both $u_{i,j}$ and $u_{i+2a,j}$ or it swaps them. 
Using also that $u_{i,j}$ and $u_{i+2a,j}$ form an antipodal pair in the rim cycle 
consisting of the vertices in $S_2$, we conclude that 
$\sigma$ is either the identity permutation, or it acts on $X$ 
as the unique reflection fixing $u_{i+a,j}$ and $u_{i+3a,j}$ and this does not  
depend on the choice of the vertex $u_{i,j}$ from $S_2$.
We finish the proof by showing that the latter possibility leads to a contradiction. 

Indeed, then  for every $i \in \Z_n$ and $j \in \Z_2$, 
$$
\sigma(v_{i,j})=\begin{cases} v_{i+a,j} & \text{if}~i \equiv j\!\!\!\!\pmod 2 \\ 
v_{i-a,j} & \text{if}~i \not\equiv j\!\!\!\!\pmod 2. 
\end{cases}
$$
As $r$ is even and $v_{0,0} \sim v_{r,1}$,   
we get $v_{a,0}=\sigma(v_{0,0}) \sim \sigma(v_{r,1})=v_{r-a,1}$. This shows 
that $r-2a \equiv \pm r \!\! \pmod n$. Since $2a \not\equiv 0\!\! \pmod n$, it follows that $2r \equiv 2a\!\! \pmod n$. 
Using also that $a$ is odd and $4a \equiv 0 \!\!\pmod {n}$, we find that $4$ divides $n$, and hence  
$2r \equiv 2a\!\!\pmod 4$.
This contradicts the assumption that $a$ is odd and $r$ is even.
\end{proof}
\subsection{Edge-transitive graphs}
The lemma below can be retrieved from the proof of \cite[Lemma~3.3]{KKM}.

\begin{lem}\label{not 3-AT}
Let $\G=R_n(a,r)$ be a non-bipartite graph such that $n> 4$ and 
$\cdc(\G)$ is edge-transitive. Then $\cdc(\G)$ is $s$-transitive   
for $s \le 2$.
\end{lem}
\subsection{Quotient graphs}
Let $\Sigma$ be an arbitrary graph and $\pi=\{C_1,\ldots,C_k\}$ be a partition of 
$\vertex(\Sigma)$.  The {\em quotient graph} $\Sigma/\pi$ is defined to have the 
vertex set $\pi$ and edges 
$\{ C_i, C_j\}$, where $C_i \ne C_j$ and there is an edge $\{u,v\} \in \edge(\Sigma)$ with 
$u \in C_i$ and $v \in C_j$. 
In the case when $\pi$ is formed by the orbits under a subgroup 
$H \le \aut(\Sigma)$, we also write $\Sigma/H$ for $\Sigma/\pi$. 
\medskip

For a divisor $k$ of $n$, let $\phi_{n,k}: \Z_n \to \Z_k$ be the group 
homorphism for which $\phi_{n,k}(1)=1$ (here the integer $1$ represents  
an element in $\Z_n$ in the left side, whereas it represents an element in $\Z_k$ in the right side).
Furthermore, let $\phi^*_{n,k}$ be the mapping from 
$\{u_{i,j}, v_{i,j} : i \in \Z_n, j \in \Z_2 \}$ to $\{u_{i,j}, v_{i,j} : i \in \Z_k, j \in \Z_2 \}$ defined  
as follows:  
\begin{equation}\label{eq:phi*}
\phi^*_{n,k}(u_{i,j})=u_{\phi_{n,k}(i),j}~\text{and}~
\phi^*_{n,k}(v_{i,j})=v_{\phi_{n,k}(i),j},~i \in \Z_n,~j \in \Z_2.
\end{equation}
The following lemma can be easily derived after one applies well-known classical 
results about quotients of arc-transitive graphs (see, e.g., \cite{L}).

\begin{lem}\label{cyclic block}
Let $\G=R_n(a,r)$, $G=\aut(\cdc(\G))$, and let $\rho \in G$ be the automorphism defined in \eqref{eq:rho}. Suppose that $\cdc(\G)$ is connected, edge-transitive, 
and that $\orb(H,\vertex(\cdc(\G)))$ is a block system for $G$ for some subgroup 
$H < \sg{\rho}$, $H \ne 1$. 
\begin{enumerate}[{\rm (i)}] 
\item If $|H|=n/2$, then $G_{u_{0,0}}$ is a $2$-group and 
$\cdc(\G)/H  \cong C_8$. 
\item If $|H| < n/2$, then $H \lhd G$ and 
$$
\cdc(\G)/H \cong \cdc(R_k(a',r')), 
$$
where $k=n/|H|$, $a'=\phi_{n,k}(a)$ and $r'=\phi_{n,k}(r)$. 
\item With the notation in (ii), the 
mapping $\phi^*_{n,k}$ defined in \eqref{eq:phi*} is a homomorphism from 
$\cdc(\G)$ onto $\cdc(R_k(a',r'))$. 
\item Letting $\Sigma=\cdc(R_k(a',r'))$ in (ii), there is an epimorphism 
$$
\varphi : G/H \to \aut(\Sigma)
$$ 
such that $\varphi(\rho)$ is the automorphism of $\Sigma$ defined in 
\eqref{eq:rho}. Furthermore,  $\Sigma$ is edge-transitive. 
\end{enumerate}
\end{lem}
\section{Proof of case (i) of Theorem~\ref{main}}\label{sec:proof1}
Let $\G=R_n(a,r)$ be a non-trivially unstable Rose Window graph and assume that $n$ is odd. We show below that $\G$ is isomorphic to a graph in W3. 

The key idea is to represent $\cdc(\G)$ as a Rose Window graph.
Let $\theta : \Z_n \times \Z_2 \to \Z_{2n}$ be the unique isomorphism mapping 
$(1,1)$ to $1$. Define the mapping $f : \vertex(\cdc(\G)) \to 
\{u_i,v_i : i \in \Z_{2n}\}$ as follows:  
$$
f(u_{i,j})=u_{\theta((i,j))}~\text{and}~f(v_{i,j})=v_{\theta((i,j))+n},~i \in \Z_n,~j \in \Z_2.
$$

A straightforward check shows that $f$ is an isomorphism from $\cdc(\G)$ to 
$R_{2n}(a',r')$, where 
$a'=\theta((a,0))$ and $r'=\theta((r,1))$. Equivalently,  
\begin{equation}\label{eq:a'r'}
a'=\begin{cases} a & \text{if $a$ is even}, \\ a+n & \text{if $a$ is odd}, \end{cases}\quad
r'=\begin{cases} r+n & \text{if $r$ is even}, \\ r & \text{if $r$ is odd}. \end{cases}
\end{equation}
Note that $a'$ is always even and $r'$ is always odd. 

Let $G=\aut((R_{2n}(a',r')))$ and denote by $\ell$ the number of edge-orbits 
under $G$. Then $\ell \in \{1,2,3\}$. 

Assume first that $\ell=3$. Since $a'$ is even, $a' \ne n$. Thus 
$|G|=4n$ due to Proposition~\ref{DKM1}. This contradicts the assumption 
that $R_n(a,r)$ is unstable. 

Now assume that $\ell=2$.
As $a' \ne n$,  $|G|=8n$ due to Proposition~\ref{DKM2}. Furthermore, by 
Proposition~\ref{DKM3}, $(r')^2 \equiv 1\!\!\pmod {2n}$ and $r'a' \equiv \pm a'\!\!\pmod {2n}$. Since $r' \equiv r\!\!\pmod n$ and $a' \equiv a\!\!\pmod n$, it follows that 
$r^2 \equiv 1\!\!\pmod n$ and $ra \equiv \pm a\!\! \pmod n$.
If $R_n(a,r)$ were edge-transitive, then $R_{2n}(a',r')$ would be edge-transitive as well, 
which is not the case. After applying Propositions~\ref{DKM2} and \ref{DKM3} to 
$R_n(a,r)$, we obtain that $|\aut(R_n(a,r))|=4n$. 
This contradicts the assumption that $R_n(a,r)$ is unstable.

Thus $\ell=1$, i.e., $R_{2n}(a',r')$ is edge-transitive.   
Let 
$$
a''=\min(a',2n-a')~\text{and}~r''=\min(r',2n-r'),
$$ 
and put $\Sigma=R_{2n}(a'',r'')$. By Proposition~\ref{iso}, 
$\Sigma \cong R_{2n}(a',r')$, in particular, $\Sigma$ is edge-transitive.
We obtain that $\Sigma$ belongs to one of the families (a)--(d) in Proposition~\ref{KKM1}.

The graph $\Sigma$ is not in the family (a). For otherwise, 
$(a'',r'')=(2,1)$. If $a$ is even, then $a'=a < n$, hence $a''=a'=a$, so $a=2$; if 
$a$ is odd, then $a'=a+n$, hence $a''=2n-a'=n-a$, and so $a=n-2$. A similar argument 
yields that $r=1$ or $n-1$. 
If $a=2$, then the vertices $u_0$ and $v_{-1}$ have the same neighborhood; and if $a=n-2$, then $u_0$ and $v_1$ share the same neighborhood. 
In either case, $\G=R_n(a,r)$ has two vertices having the same neighborhood, which is impossible. 

The graph $\Sigma$ is not in the family (b). For otherwise, $r''=n-1$, which is impossible as $r''$ is odd. Also, $\Sigma$ is neither in the family (c) because $n$ is odd.

Thus $\Sigma=R_{2n}(a'',r'')$ is in the family (d), i.e., 
$a''=2d$ such that $d^2 \equiv \pm 1\!\!\pmod n$,  $2 \le  a'' \le n$ and
$r'' \in \{1,n-1\}$ is odd. It follows that $r''=1$. We obtain 
$a=2d$ or $n-2d$ and $r=1$ or $n-1$.
Finally, $d \not\equiv \pm 1\!\!\pmod n$, for otherwise, $R_n(a,r)$ has two vertices 
with the same neighborhood, and $R_n(a,r)$ is indeed isomorphic to a graph in W3.
\section{Proof of case (ii) of Theorem~\ref{main}}\label{sec:proof2}
For this section we set the following assumptions. 

\begin{hypo}\label{hypo1}
$\G=R_n(a,r)$ is a non-trivially unstable graph such that   
$n=2m$, $1 \le a, r \le m$ and $\cdc(\G)$ is not edge-transitive.  
Furthermore, 
\begin{quote}
$G=\aut(\cdc(\G))$, \\ [+1ex] 
$C=\sg{\rho}$, where $\rho$ is defined in \eqref{eq:rho}, \\ [+1ex]
$\sigma$ is an unexpected automorphism of $\cdc(\G)$. 
\end{quote}
\end{hypo}

To settle case (ii) of Theorem~\ref{main}, we have to show that under the above assumptions $\G$ is isomorphic to a graph 
in one of W1 and W4--W9. This will be done in three lemmas.

\begin{lem}\label{L1}
Assuming Hypothesis~\ref{hypo1}, suppose that 
$\sigma$ maps $S_1 \cup S_2$ to itself. Then 
$\G$ is isomorphic to a graph in one of W1, W5 and W6. 
\end{lem}
\begin{proof}
Since $\sg{\rho,\beta}$ is transitive on $S_1 \cup S_2$, 
$\sigma(u_{0,0})=\alpha(u_{0,0})$ for some 
$\alpha \in \sg{\rho,\beta}$, where $\beta$ is the automorphism of $\cdc(\G)$ 
defined in \eqref{eq:beta}.   
As $\alpha^{-1}\sigma$ is unexpected, we may assume   
that $\sigma(u_{0,0})=u_{0,0}$. 

For $\alpha \in G$ such that $\alpha(S_1)=S_1$, 
let $\alpha^*$ denote the permutation of $S_1$ induced by the restriction of 
$\alpha$ to $S_1$.  
Clearly, $\alpha^*$ is an automorphism of the rim cycle containing 
$u_{0,0}$. Thus $\sigma^* \in \sg{\mu^*}$, hence we may assume  
that $\sigma \in G_{(S_1)}$. Here $\mu$ is the automorphism of $\cdc(\G)$ 
defined in \eqref{eq:mu}, in particular, $\mu(S_1)=S_1$ and $\mu^*$ is well-defined.

Assume first that $a$ is even. Then $r$ is even due to Lemma~\ref{bipartite}. 
If $a=m$, then $\G$ is in W1. 
Let $a \ne m$. 
Suppose that $4r \not\equiv 0\!\! \pmod n$. Then Lemma~\ref{pointwise}(i) yields that $\sigma$ is the identity, which is impossible, hence 
$4r \equiv 0\!\!\pmod n$. 

Assume for the moment that $r=m$. Denote by $\ell_s$ the length of the spoke cycles of $\cdc(\G)$.  
Since $a \ne m$, it follows that $\ell_s \geq 6$ according to~\eqref{eq:ell}. Let $X$ be a spoke cycle. 
Then $|\vertex(X) \cap (S_3 \cup S_4)|=\frac{1}{2}\ell_s$ and
\[
|\vertex(X) \cap S_1|=\frac{1}{2}\ell_s~\text{or}~|\vertex(X) \cap S_2|=\frac{1}{2}\ell_s. 
\]
Let us call $X$ of type $1$ in the former case and 
of type $2$ in the latter case. It is clear that $\sigma$ fixes all vertices of a spoke cycle of type $1$. Furthermore, the hub edges 
$\{v_{i,j},v_{i+m,j+1}\}$, where $i \in \Z_n, j \in \Z_2$, form a matching, and the vertices $v_{i,j}$ and $v_{i+m,j+1}$ 
belong to spoke cycles of distinct types. This shows that $\sigma$ must fix every vertex in $S_3 \cup S_4$, and consequently, 
also in $S_2$. This contradicts the assumption $\sigma$ is not the identity permutation, and thus $r \ne m$.
Using also that $r$ is even and $4r \equiv 0\!\!\pmod n$, we conclude that $r=\frac{n}{4}$ and $n$ is divisible by $8$, hence
that $\G$ is in W5.

Assume second that $a$ is odd. 
Then $4a \equiv 0\!\!\pmod n$ follows from Lemma~\ref{pointwise}(ii) and the condition that $\sigma \in G_{(S_1)}$.
We show below that $a \ne m$. 
On the contrary, assume that $a=m$. Fix an arbitrary vertex 
$v_{i,j}$ and let $X$ be the spoke cycle containing $v_{i,j}$. Then 
$$
X=\big(\, v_{i,j},u_{i,j+1},v_{i+m,j},u_{i+m,j+1}\, \big).
$$
This shows that $\sigma$ fixes both $u_{i,j+1}$ and $u_{i+m,j+1}$, and 
$\sigma(v_{i,j})=v_{i,j}$ or $v_{i+m,j}$. In view of Proposition~\ref{HM}, we may assume 
that 
\begin{equation}\label{eq:Bij}
\sigma(v_{i,j})=v_{i,j},~\sigma(v_{i+m,j})=v_{i+m,j},~\sigma(v_{i,j+1})=
v_{i+m,j+1}~\text{and}~\sigma(v_{i+m,j+1})=v_{i,j+1}.
\end{equation}

Let $Y$ be the hub cycle containing $v_{i,j}$. Then $\sigma(Y)=Y$, and thus  
$\sigma(v_{i',j'})=v_{i',j'}$, where $v_{i',j'}$ is the vertex antipodal to $v_{i,j}$ in $Y$. 
Then $v_{i',j'}=v_{i,j+1}$ if $r$ has odd order in $\Z_n$, and 
$v_{i',j'}=v_{i+m,j+1}$ otherwise. In either case we obtain a contradiction with  
the assumptions in \eqref{eq:Bij}, and by this we have shown that $a \ne m$.
Finally,  Lemma~\ref{pointwise}(iii) shows that $r$ is odd, and so $\G$ is in W6.
\end{proof}

For the rest of the section it will be assumed that 
$\sigma(S_1 \cup S_2) \ne S_1 \cup S_2$. Consequently, $G$ is transitive on 
$\vertex(\cdc(\G))$ and divides $\edge(\cdc(\G))$ into two orbits (recall that $\cdc(\G)$ 
is assumed to be non-edge-transitive).   
There are two hub cycles of the same length $n$, so by \eqref{eq:ell},
\begin{equation}\label{eq:r}
\gcd(r,n) \in \{1,2\},~\text{and if}~\gcd(n,r)=2,~\text{then}~m~\text{is odd}.
\end{equation}  

Denote by $X_1$ and $X_2$ the two 
rim cycles and by $X_3$ and $X_4$ the two hub cycles so that 
$\vertex(X_i)=S_i$ for $i=1,2$ and $v_{0,0} \in \vertex(X_3)$,  
where $S_1$ and $S_2$ are defined in \eqref{eq:S12}. 
If $r$ is even, then 
\begin{equation}\label{eq:r-even}
\vertex(X_3)=\{v_{i,j} : i \equiv 0\!\!\!\!\pmod 2\}~\text{and}~
\vertex(X_4)=\{v_{i,j} : i \equiv 1\!\!\!\!\pmod 2\};
\end{equation}
and if $r$ is odd, then 
\begin{equation}\label{eq:r-odd}
\vertex(X_3)=S_3~\text{and}~\vertex(X_4)=S_4.
\end{equation} 

Clearly, the subsets
\begin{equation}\label{eq:bsystem}
S_1=\vertex(X_1),~S_2=\vertex(X_2),~\vertex(X_3),~\vertex(X_4)
\end{equation}
form a block system for $G$. 

\begin{lem}\label{L2} 
Assuming Hypothesis~\ref{hypo1}, suppose that 
$\sigma(S_1 \cup S_2) \ne S_1 \cup S_2$ and let $K$ be the kernel of 
the action of $G$ on the block system defined in \eqref{eq:bsystem}.
If $K_{(S_1)} \ne 1$, then $\G$ is isomorphic to a graph in one of 
W1, W5 and W6.
\end{lem}
\begin{proof}
Fix $ \tau \in K_{(S_1)}$ such that $\tau$ is not the identity permutation. 
Clearly, $\tau$ fixes $S_1$ pointwise and $S_2$ setwise.  

Thus if $a$ is even, then the proof of Lemma~\ref{L1} can be repeated with 
$\tau$ playing the role of $\sigma$.

Let  $a$ be odd. If $a=m$, then every spoke cycle contains exactly one vertex from 
$S_1$, $S_2$, $V(X_3)$ and $V(X_4)$, respectively, implying that $\tau$ is the 
identity permutation, which is impossible. Hence $a \ne m$, and  
the argument used in the proof of Lemma~\ref{L1} can be applied again. 
\end{proof}

\begin{lem}\label{L3} 
Assuming Hypothesis~\ref{hypo1}, suppose that 
$\sigma(S_1 \cup S_2) \ne S_1 \cup S_2$ and let $K$ be the kernel of 
the action of $G$ on the block system defined in \eqref{eq:bsystem}.
If $K_{(S_1)}=1$, then $\G$ is isomorphic to a graph in one of W7, W8 and W9.
\end{lem}
\begin{proof}
Since $\sg{\rho,\beta}$ is transitive on $S_3 \cup S_4$, 
$\sigma(u_{0,0})=\alpha(v_{0,1})$ for some 
$\alpha \in \sg{\rho,\beta}$, where $\beta$ is the automorphism of $\cdc(\G)$ 
defined in \eqref{eq:beta}.   As $\alpha^{-1}\sigma$ is unexpected, we may assume that $\sigma(u_{0,0})=v_{0,1}$. Then the hub cycle containing $v_{0,1}$ 
is not preserved by $\sigma$, hence we have that $\sigma(v_{0,1})=u_{0,0}$ or 
$u_{a,0}$. Replacing $\sigma$ with $\rho^a\mu\sigma$ in the latter case,  we may also assume that $\sigma$ 
interchanges $u_{0,0}$ and $v_{0,1}$. 

Furthermore, replacing $\G$ with $R_n(a,-r)$ if necessary (see Proposition~\ref{iso}), 
we may also assume that $\sigma(u_{1,1})=v_{r,0}$ holds. Using that $\sigma$ is an isomorphism from 
$X_1$ to the hub cycle containing $v_{0,1}$, which maps 
the arc $(u_{0,0},u_{1,1})$ to the arc $(v_{0,1},v_{r,0})$, we deduce that 
\begin{equation}\label{eq:1}
\sigma(u_{i,j})=v_{ri,j+1}~\text{for every}~u_{i,j} \in S_1.
\end{equation}

Since $K_{(S_1)}=1$, it follows that $K$ is isomorphic to a subgroup of $D_{2n}$. 
It follows from \eqref{eq:r-even} and 
\eqref{eq:r-odd} that $\rho^2 \in K$. 

Suppose that $n=4$, so $1 \le a,r \le 2$. Then $r=1$ by \eqref{eq:r}. 
Since $\G$ has no two vertices with the same neighborhood, it follows 
that $a=1$, and so $\G$ is isomorphic to the graph in W9 with 
$m=2$, $a=1$ and $s=3$ due to Proposition~\ref{iso}. 

For the rest of the proof we assume that $n > 4$. Then $\sg{\rho^2}$ is 
characteristic in $K$. Using also that $K \triangleleft G$, we obtain that 
$\sg{\rho^2} \lhd G$.  This implies that there is an integer $k$ such that 
$\gcd(k,m)=1$ and $\sigma \rho^2=\rho^{2k} \sigma$.  
Thus 
$$
\sigma(u_{2,0})=\sigma\rho^2(u_{0,0})=
\rho^{2k}\sigma(u_{0,0})=v_{2k,1}. 
$$
On the other hand, $\sigma(u_{2,0})=v_{2r,1}$ due to \eqref{eq:1}, hence 
$2k \equiv 2r \pmod n$, $k \equiv r \pmod m$,  and therefore, we may assume   
that $k=r$.  Note that, if $i$ is even, then $\sigma \rho^i \sigma^{-1}=
(\sigma \rho^2 \sigma^{-1})^{\frac{i}{2}}=\rho^{ri}$.  
The following identity holds. 
\begin{equation}\label{eq:2}
\sigma \rho^{i}=\rho^{ri}\sigma~\text{for every}~i \in \Z_n,~i \equiv 0\!\!\!\!\pmod 2.
\end{equation}

\noindent{\bf Case~1.} $r$ is even.
\medskip

By \eqref{eq:r}, $m$ is odd, so $u_{m,1}$ is the vertex antipodal to $u_{0,0}$ 
in $X_1$. Then $\sigma^2(u_{m,1})=u_{m,1}$. Using also \eqref{eq:1}, we 
obtain that $\sigma(v_{0,0})=\sigma^2(u_{m,1})=u_{m,1}$. Thus by 
\eqref{eq:2}, $\sigma(v_{r,0})=\sigma \rho^r(v_{0,0})=\rho^{r^2}\sigma(v_{0,0})=
u_{r^2+m,1}$. On the other hand,  
$\sigma(v_{r,0}) \sim \sigma(v_{0,1})=u_{0,0}$, implying that  
 $r^2 \equiv \pm 1 \pmod m$. 
 
Assume first that $a$ is odd. 
If $a=m$, then $\G$ belongs to W8 if $r \not\equiv \pm 1 \pmod m$, 
and it belongs to W7 if $r \equiv \pm 1 \pmod m$. 
Assume that $a < m$. Observe that $\{u_{a,1}\}=\cdc(\G)(v_{0,0}) \cap S_1$. 
Since $r$ is even, $v_{0,1} \in V(X_3)$, see \eqref{eq:r-even}, hence 
$\sigma(S_1)=V(X_3)$. Using also \eqref{eq:1} and \eqref{eq:2}, 
we have that  
$$
\{v_{ra,0}\}=\{\sigma(u_{a,1})\}=
\cdc(\G)(\sigma(v_{0,0})) \cap \vertex(X_3)=\cdc(\G)(u_{m,1}) \cap 
\vertex(X_3)=\{v_{m-a,0}\}.
$$
This shows that $ra \equiv m-a\!\!\pmod n$. 
Then $r^2 a \equiv -ra \equiv m+a\!\!\pmod n$, and using also that 
$r^2 \equiv \pm 1\!\! \pmod m$ and $a < m$, we deduce that 
$r^2 \equiv 1\!\!\pmod m$. Set $s=m+r$. 
Then $\G=R_{2m}(a,m+s)=R_{2m}(a,m-s)$, 
$s^2 \equiv 1\!\!\pmod n$ and $sa \equiv -a\!\!\pmod n$, hence 
$\G$ belongs to W9. 

Now let $a$ be even. Then $u_{0,0}=\sigma(v_{0,1}) \sim \sigma(u_{a,0})=v_{ra,1}$, 
implying that $ra \equiv -a\!\!\pmod n$. It follows that $r^2 \equiv 1\!\!\pmod m$, and 
taking $s=m+r$, we find that $\G$ belongs to W9. 
\medskip

\noindent{\bf Case~2.} $r$ is odd. 
\medskip

Then $\vertex(X_3)=S_3$ and $\vertex(X_4)=S_4$, where $S_3$ and $S_4$ are 
the sets defined in \eqref{eq:S34}, see \eqref{eq:r-odd}.
By Lemma~\ref{bipartite}, $a$ is odd.  Consider the path 
$$
\mathbf{P}=(v_{-a,0},u_{0,1},v_{0,0},u_{a,1},v_{a,0})
$$ 
and its image $\sigma(\mathbf{P})$, see Figure~\ref{fig1}. 
The vertices of $\mathbf{P}$ (from the left to the right) 
belong to the sets $S_4, S_2, S_3, S_1$ and $S_4$, respectively. Therefore, 
the their images are from  $S_1, S_3, S_2, S_4$ and $S_1$, respectively. 
Furthermore, by \eqref{eq:1}, $\sigma(u_{a,1})=v_{ra,0}$, so $x_4=ra$. All these imply in turn that, $x_5=ra$, $x_3=ra+a$, $x_2=ra+a$ and $x_1=ra+2a$. 
Using \eqref{eq:2}, we find that 
\begin{eqnarray*}
u_{ra+2a,1} &=& \sigma(v_{-a,0})=\sigma \rho^{-a-1}(v_{1,0})=\rho^{-ra-r}
\sigma(v_{1,0}) \\ 
u_{ra,1} &=& \sigma(v_{a,0})=\sigma \rho^{a-1}(v_{1,0})=\rho^{ra-r}
\sigma(v_{1,0}).
\end{eqnarray*}
Thus $2(ra+a) \equiv 0\!\!\pmod n$, whence $ra+a=0$ or $m$ in $\Z_n$.

\begin{figure}[t]
\begin{center}
\begin{tikzpicture}[scale=1.5] ---
\fill (0,0) circle (1.5pt);
\fill (1,0) circle (1.5pt);
\fill (2,0) circle (1.5pt);
\fill (3,0) circle (1.5pt);
\fill (4,0) circle (1.5pt);
\fill (0,1) circle (1.5pt);
\fill (1,1) circle (1.5pt);
\fill (2,1) circle (1.5pt);
\fill (3,1) circle (1.5pt);
\fill (4,1) circle (1.5pt);
\draw (0,0) node[below]{\footnotesize $u_{x_1,1}$};
\draw (1,0) node[below]{\footnotesize $v_{x_2,0}$};
\draw (2,0) node[below]{\footnotesize $u_{x_3,1}$};
\draw (3,0) node[below]{\footnotesize $v_{x_4,0}$};
\draw (4,0) node[below]{\footnotesize $u_{x_5,1}$};
\draw (0,1) node[above]{\footnotesize $v_{-a,0}$};
\draw (1,1) node[above]{\footnotesize $u_{0,1}$};
\draw (2,1) node[above]{\footnotesize $v_{0,0}$};
\draw (3,1) node[above]{\footnotesize $u_{a,1}$};
\draw (4,1) node[above]{\footnotesize $v_{a,0}$};
\draw (0,0) -- (1,0) -- (2,0) -- (3,0) -- (4,0); 
\draw (0,1) -- (1,1) -- (2,1) -- (3,1) -- (4,1);
\draw (-0.25,0.5) node{\footnotesize $\sigma$};
\draw (-0.05,0.7) -- (0.05,0.7);
\draw (0.95,0.7) -- (1.05,0.7);
\draw (1.95,0.7) -- (2.05,0.7);
\draw (2.95,0.7) -- (3.05,0.7);
\draw (3.95,0.7) -- (4.05,0.7);
\draw[->] (0,0.7) -- (0,0.3);
\draw[->] (1,0.7) -- (1,0.3);
\draw[->] (2,0.7) -- (2,0.3);
\draw[->] (3,0.7) -- (3,0.3);
\draw[->] (4,0.7) -- (4,0.3);
\end{tikzpicture}
\caption{The path $\mathbf{P}$ and its image under $\sigma$.} 
\label{fig1}
\end{center}
\end{figure}
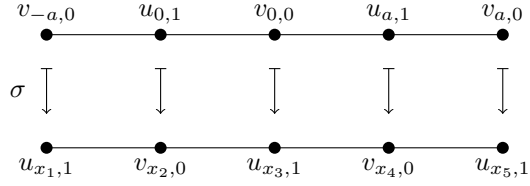

Assume first that $ra+a \equiv 0\!\!\pmod n$. We show below that this possibility cannot occur. 
Notice that, then  
\begin{equation}\label{eq:sigma1}
\sigma(u_{0,1})=v_{0,0},~\sigma(v_{0,0})=u_{0,1}~\text{and}~
\sigma(v_{1,0})=u_{r,1}.
\end{equation}

By the initial assumptions on $\sigma$ we also have
\begin{equation}\label{eq:sigma2}
\sigma(u_{0,0})=v_{0,1},~\sigma(u_{1,1})=v_{r,0}~\text{and}~\sigma(v_{0,1})=u_{0,0}.
\end{equation}

Now $\cdc(\G)(v_{0,1}) \cap S_2=\{u_{a,0}\}$ yields 
$$
\{\sigma(u_{a,0})\}=\cdc(\G)(\sigma(v_{0,1})) \cap S_3=
\cdc(\G)(u_{0,0}) \cap S_3=v_{-a,1},
$$
from which it follows that $v_{-a,1}=\sigma(u_{a,0})=\sigma \rho^{a-1}(u_{1,0})=\rho^{ra-r}\sigma(u_{1,0})$, so $\sigma(u_{1,0})=v_{r,1}$. 
It can be shown in the same way that $\sigma(v_{-a,1})=u_{a,0}$, and deduce from 
this that $\sigma(v_{1,1})=u_{r,0}$. To sum up,
\begin{equation}\label{eq:sigma3}
\sigma(u_{1,0})=v_{r,1}~\text{and}~\sigma(v_{1,1})=u_{r,0}.
\end{equation}

The conditions in \eqref{eq:sigma1}--\eqref{eq:sigma3} together with \eqref{eq:2} yield that $\sigma$ commutes with $\beta$. This, however, contradicts the assumption that 
$\sigma$ is unexpected due to Lemma~\ref{cor-HM}. We verify that $\sigma \beta(u_{i,j})=\beta \sigma(u_{i,j})$ 
and $\sigma\beta(v_{i,j})=\beta \sigma(v_{i,j})$
for the case when $i$ is even and $j=0$ only, the remaining cases can 
be handled in the same way. 
\begin{eqnarray*}
\sigma \beta(u_{i,0}) &=& \sigma(u_{i,1})=\sigma \rho^{i}(u_{0,1})=\rho^{ri} 
\sigma(u_{0,1})=\rho^{ri} (v_{0,0})=\beta \rho^{ri}(v_{0,1}) \\ 
&=& 
\beta \rho^{ri} \sigma(u_{0,0})=\beta \sigma \rho^{i}(u_{0,0})=\beta \sigma (u_{i,0}).
\end{eqnarray*}
and
\begin{eqnarray*}
\sigma \beta(v_{i,0}) &=& \sigma(v_{i,1})=\sigma \rho^{i}(v_{0,1})=\rho^{ri} 
\sigma(v_{0,1})=\rho^{ri} (u_{0,0})=\beta \rho^{ri}(u_{0,1}) \\ 
&=& 
\beta \rho^{ri} \sigma(v_{0,0})=\beta \sigma \rho^{i}(v_{0,0})=\beta \sigma (v_{i,0}).
\end{eqnarray*}

We have shown that $ar+a \not\equiv 0\!\!\pmod n$, and therefore, $ar
\equiv m-a\!\!\pmod n$. It follows that $m$ is even, in particular, $a \ne m$. 
Then $u_{0,0}=\sigma(v_{0,1}) \sim \sigma(v_{r,1})=\rho^{r^2-r}\sigma(v_{1,1})=
u_{r^2+m,1}$, implying that $r^2 \equiv \pm 1\!\!\pmod m$. 
Using also that $ra \equiv -a\!\!\pmod m$ and $a \ne m$, we have 
that $r^2 \equiv 1\!\!\pmod m$.
Finally, taking $s=m+r$, we find that $\G$ belongs to W9. 
\end{proof}
\section{Proof of case (iii) of Theorem~\ref{main}}\label{sec:proof3}
For this section we set the following assumptions. 

\begin{hypo}\label{hypo2}
$\G=R_n(a,r)$ is a non-trivially unstable graph, $n=2m$, $1 \le a, r \le m$ and $\cdc(\G)$ is edge-transitive.  Furthermore, 
\begin{quote}
$G=\aut(\cdc(\G))$,  \\ [+1ex] 
$G^+$ is the subgroup of $G$ preserving the biparts of $\cdc(\G)$, \\ [+1ex] 
$C=\sg{\rho}$, where $\rho$ is defined in \eqref{eq:rho}, \\ [+1ex]
$L=\sg{\rho^m}$.
\end{quote}
\end{hypo}

Note that, if $n > 4$, then by Lemma~\ref{not 3-AT}, $\cdc(\G)$ is $s$-transitive for 
$s \le 2$, in particular, it is vertex-transitive. These properties will be used a couple of times in the section. 
\medskip

In order to settle case (iii) of Theorem~\ref{main}, we have to 
show that under the above assumptions $\G$ is isomorphic to 
a graph in one of W2 and W4.  We deal first with the case when $C \lhd G$. 

\begin{lem}\label{L4}
Assuming Hypothesis~\ref{hypo2}, suppose that $C \triangleleft G$. Then $m$ is odd and 
$\G$ is isomorphic to a graph in the family W4.
\end{lem}
\begin{proof}
Fix $\sigma \in G_{u_{0,0}}$ such that $\sigma(u_{1,1})=v_{0,1}$. 
There is an integer $k$ such that $\gcd(k,n)=1$ and 
$\sigma \rho^i =\rho^{ki}\sigma$ for every $i \in \Z_n$. Then 
$$
\sigma(u_{-1,1})=\sigma\rho^{-2}(u_{1,1})=\rho^{-2k}\sigma(u_{1,1})=v_{-2k,1}. 
$$
This implies that $v_{-2k,1}=v_{-a,1}$, hence $a \equiv 2k\!\!\pmod n$. 
Note that $a$ is even.  Then $\sigma(v_{0,1})=u_{1,1}$ or $u_{-1,1}$. 
Here we consider only the first case as the other case can be handled in the same way. 
For every $i \in \Z_n$, 
\begin{eqnarray*}
\sigma(u_{i,0}) &=& \sigma\rho^{i}(u_{0,0})=\rho^{ki}\sigma(u_{0,0})=u_{ki,0}, \\
\sigma(u_{i,1}) &=& \sigma\rho^{i-1}(u_{1,1})=\rho^{ki-k}\sigma(u_{1,1})=v_{ki-k,1}, \\
\sigma(v_{i,1}) &=& \sigma\rho^{i}(v_{0,1})=\rho^{ki}\sigma(v_{0,1})=u_{ki+1,1}. 
\end{eqnarray*}
Note that $\orb_C(v_{0,0})$ is mapped to itself by $\sigma$. 
Thus $u_{-kr+1,1}=\sigma(v_{-r,1})$ and 
$u_{kr+1,1}=\sigma(v_{r,1})$ share a common neighbor in 
$\orb_C(v_{0,0})$. This implies that 
$2kr \equiv \pm a \equiv \pm 2k\!\!\pmod n$, whence 
$r \equiv \pm 1 \pmod m$. Since $a$ is even, $r$ must be even 
as well because $\G$ is not bipartite. Thus $m$ is odd and 
$r=m-1$ or $m+1$. We conclude that $\G$ is isomorphic to a graph in 
the family W4.
\end{proof}

The case when $C$ is not normal in $G$ requires three preparatory lemmas. 
We denote the center of $G$ by $Z(G)$, and for a subgroup $H \le G$ and an element $g \in G$, 
by $C_G(H)$ and $C_G(g)$ the centralizers of $H$ and $g$, respectively, in $G$. 

\begin{lem}\label{L5}
Assuming Hypothesis~\ref{hypo2}, $L \triangleleft G$, consequently, 
$\rho^m \in Z(G)$.
\end{lem}
\begin{proof}
For the sake of simplicity, let $V=\vertex(\cdc(\G))$. 
We proceed by induction on $n$. 
With the help of the computer algebra package 
{\sc Magma}~\cite{BCP} we checked that the lemma holds if $n \le 48$. 
From now on it is assumed that $n > 48$.
\medskip

\noindent{\bf Case~1.} $G_{u_{0,0}}$ is not a $2$-group.
\medskip

As the graph $\cdc(\G)$ is $s$-arc-transitive for $s \le 2$, see the remark after Hypothesis~\ref{hypo2}, it follows from Proposition~\ref{s-AT} that $G_{u_{0,0}} 
\cong A_4$ or $S_4$. Thus $|G^+| \le 48 n  < |C|^2$ because of  
$n > 48$. By Proposition~\ref{L}, $G^+$ contains a normal subgroup 
$N$ such that $1 < N  \le C$. If $|N|$ is even, then $L$ is a characteristic 
subgroup of $N$, and we find that $L \triangleleft G$. 

Assume now that $|N|$ is odd. 
Let $\B$ be a minimal non-trivial block system 
for $G$, which is a refinement of the block system $\orb(N,V)$. 
It is easy to show that $\B=\orb(H,V)$ for some subgroup $1 < H \le N$. 

If $|H|=n/2$, then $G_{u_{0,0}}$ is a $2$-group due to Lemma~\ref{cyclic block}(i), 
which is impossible. Thus $|H| < n/2$, and by part (ii) of the same lemma, 
$$
\cdc(\G)/H \cong \cdc\big(R_k(\phi_{n,k}(a),\phi_{n,k}(r))\big),
$$ 
where $k=n/|H|$. As $|H|$ is odd, $k$ is even and the induction 
hypothesis can be applied to $\cdc\big(R_k(\phi_{n,k}(a),\phi_{n,k}(r))\big)$. 
Lemma~\ref{cyclic block}(iii) shows that $HL/H \triangleleft G/H$, whence 
$HL \triangleleft G$. This implies that $L \triangleleft G$, as claimed. 
\medskip

\noindent{\bf Case~2.} $G_{u_{0,0}}$ is a $2$-group.
\medskip

Let $T$ be a Sylow $2$-subgroup of $G^+$ such that $G_{u_{0,0}} \le T$. 
Then $\orb_T(u_{0,0})$ is a block for $G$ (see \cite[Theorem~1.5A]{DM}). 
Let $\Delta$ be a minimal block contained in the latter orbit such that 
$u_{0,0} \in \Delta$ and $\Delta \ne \{u_{0,0}\}$, and let $\B$ be the block system generated by $\Delta$. As $|\Delta|$ divides $|\orb_T(u_{0,0})|$, 
$|\Delta|$ is a power of $2$. Since $G_{u_{0,0}}$ is a 
$2$-group, we obtain that $\Delta$ contains a vertex $v$ such that 
$v \ne u_{0,0}$ and $G_{u_{0,0}}=G_v$. 
It is known that the subset $\Delta'$ of $V$, defined as 
$$
\Delta'=\big\{w : w \in V~\text{and}~G_w=G_{u_{0,0}} \big\}
$$
is a block for $G$ (see, e.g., \cite[Excercise~1.6.5]{DM}). Since $\Delta \cap \Delta'$ is also a block and 
$\Delta$ is minimal, we obtain $\Delta \subseteq \Delta'$. Using also 
that $\mu \in G_{u_{0,0}}$, see \eqref{eq:mu}, we obtain that $\Delta$ is one of 
the following sets:
\begin{equation}\label{eq:sets}
\{u_{0,0},u_{m,0}\},~\{u_{0,0},v_{b,0}\},~\text{and}~
\{u_{0,0},u_{m,0},v_{b,0},v_{b+m,0}\},
\end{equation}
where $a  \equiv -2b\!\!\pmod n$.
If $\Delta=\{u_{0,0},u_{m,0}\}$, then $\B=\orb(L,V)$. 
By Lemma~\ref{cyclic block}(ii), $L \triangleleft G$. Therefore, to finish the proof it 
sufficient to exclude the last two sets in \eqref{eq:sets}.

Assume first that 
$$
\Delta=\{u_{0,0},u_{m,0},v_{b,0},v_{b+m,0}\}.
$$ 
Since $\Delta$ is a minimal block for $G$, it follows that the 
setwise stabilizer $G_{\{\Delta\}}$ acts primitively on 
$\Delta$. Denote by $K$ and $(G_{\{\Delta\}})^*$ the kernel and  
the image, respectively, of the action of 
$G_{\{\Delta\}}$ on $\Delta$. We have that 
$(G_{\{\Delta\}})^* \cong G_{\{\Delta\}}/K$. 
As $G_{u_{0,0}}$ is a $2$-group, $G_{\{\Delta\}}$ is also a $2$-group. Therefore, 
there exists a subgroup $M \le G_{\{\Delta\}}$ such that $K < M$, $|M/K|=2$ and 
$M/K < Z(G_{\{\Delta\}}/K)$. We obtain that $M \lhd G_{\{\Delta\}}$ and there are two $M$-orbits on $\Delta$, 
each of size $2$. This implies that $G_{\{\Delta\}}$ is imprimitive on $\Delta$, a contradiction.  

Assume second that $\Delta=\{u_{0,0},v_{b,0}\}$.  
Consider the induced subgraph $\cdc(\G)[\Delta_1 \cup \Delta_2]$, where 
$\Delta_1, \Delta_2 \in \B$ and $\Delta_1$ and $\Delta_2$ are adjacent in  
$\G/\B$. Denote the latter subgraph by $\Sigma$. 
As $\cdc(\G)$ is arc-transitive, 
see the remark after Hypothesis~\ref{hypo2}, 
$\Sigma$ does not 
depend on the choice of $\Delta_1$ and $\Delta_2$. There are three options for 
$\Sigma$, namely, $\Sigma \cong \K_2$ or $2\K_2$ or $\C_4$.  
Choosing $\Delta_1=\Delta=\{u_{0,0},v_{b,0}\}$ and  
$\Delta_2=(\beta\rho^{-b})(\Delta)=\{u_{-b,1},v_{0,1}\}$, one can observe that 
$\Sigma \cong 2\K_2 $ or $\C_4$.  Now choosing $\Delta_2$ to be 
$(\beta\rho)(\Delta)$, we conclude that 
$\{v_{b,0},v_{b+1,1}\}$ is an edge, whence $r \equiv \pm 1\!\!\pmod n$, 
in particular, $r$ is odd. 
On the other hand, $a \equiv -2b\!\!\pmod n$, so $a$ is even. 
By Lemma~\ref{bipartite}, $\G$ is bipartite, which is a contradiction.
\end{proof}

We describe next a situation, which we will encounter a couple of times 
in the rest of the paper. Suppose that $d$ is an odd divisor of $m$ such that 
$$
d > 3,~\sg{\rho^d} \triangleleft G,~\phi_{n,d}(a)=\pm 2~\text{and}~
\phi_{n,d}(r)=\pm 1.
$$ 
Let $\Sigma=\cdc\big(R_d\big(\phi_{n,d}(a),\phi_{n,d}(r)\big)\big)$ and let 
$$
\varepsilon=\begin{cases}
1 & \text{if}~\phi_{n,d}(a)=-2, \\
-1 & \text{if}~\phi_{n,d}(a)=2.
\end{cases}   
$$
By Lemma~\ref{cyclic block}(iii), 
$\phi^*_{n,d}$ is a homomorphism from $\cdc(\G)$ onto $\Sigma$, where 
$\phi^*_{n,d}$ is defined in \eqref{eq:phi*}. Observe that, for every $i \in \Z_d$ and 
$j \in \Z_2$, $v_{i+\varepsilon,j}$ is the unique vertex of $\Sigma$ for which 
$$
\Sigma(u_{i,j})=\Sigma(v_{i+\varepsilon,j}).
$$
This shows that sets $\{u_{i,j},v_{i+\varepsilon,j}\}$, $i \in \Z_d$, $j \in \Z_2$, form 
a block system for $\aut(\Sigma)$.  Consequently, the sets 
\begin{eqnarray}
\Delta_{i,j}&:=&\big\{ (\phi^*_{n,d})^{-1}(u_{i,j}), (\phi^*_{n,d})^{-1}(v_{i+\varepsilon,j}) \big\} 
\label{eq:Delta} \\ 
&=& \big\{ u_{i+xd,j}, v_{i+xd+\varepsilon,j} : 
0 \le x \le n/d-1 \big\},~0 \le i \le d-1~\text{and}~j=0,1, \nonumber
\end{eqnarray}
form a block system $\B$ for $G$. Note that 
the quotient graph $\cdc(\G)/\B$ is a cycle of length $2d$. 
\medskip

In the next lemma we exclude certain possibilities for $(a,r)$.

\begin{lem}\label{L6} 
Assuming Hypothesis~\ref{hypo2}, suppose that $m=2l$ for some odd number $l$. 
Then 
$$
(a,r) \notin \big\{\, (l+2,l+1), (l+2,l-1), (l-2,l+1), (l-2,l-1)\, \big\}.
$$ 
\end{lem}
\begin{proof}
On the contrary, assume that $(a,r)$ is one of the pairs given in the lemma. 
By Lemma~\ref{L5}, $L \lhd G$. As $0 < a, r \le m$, 
$\phi_{n,m}(a)=a$ and $\phi_{n,m}(r)=r$, and hence 
$\cdc(\G)/L \cong \cdc(R_m(a,r))$ due to 
Lemma~\ref{cyclic block}(ii). Applying Lemma~\ref{L5} to $\cdc(R_m(a,r))$, 
we obtain $\sg{\rho^l}/L \lhd G/L$, and hence $\sg{\rho^l} \lhd G$. 

If $l=3$, then using the computer algebra package 
{\sc Magma}~\cite{BCP}, we find that if $(a,r)=(l+2,l+1)=(5,4)$ or 
$(a,r)=(l-2,l-1)=(1,2)$, then $\cdc(\G)$ is not edge-transitive; and if 
$(a,r)=(l+2,l-1)=(5,2)$ or $(a,r)=(l-2,l+1)=(1,4)$, then $\cdc(\G)$ is edge-transitive, 
but $\G$ is stable. Both cases are impossible, hence in the rest of the proof we 
assume that $l > 3$. Then $G$ admits the block system $\B$ with blocks $\Delta_{i,j}$ defined 
in \eqref{eq:Delta} with $d=l$. Let $K$ be the kernel of the action of $G$ on $\B$.  
 \medskip

\noindent{\bf Case~1.}~$(a,r)=(l+2,l+1)$ or $(l-2,l-1)$. 
\medskip

We discuss only the case when $(a,r)=(l+2,l+1)$ because the 
other case can be handled by copying the same argument. For $i \in \Z_l$ and $j \in \Z_2$,  
$$
\cdc(\G)[\Delta_{i,j} \cup \Delta_{i+1,j+1}] \cong 4\C_4.
$$ 
Clearly, the automorphisms in $K$ permute these $4$-cycles. 
Consider the $4$-cycles $X_1, X_2, X_3$ and $X_4$ shown in Figure~\ref{fig2}.  
Choose $\sigma \in G_{u_{0,0}}$ such that $\sigma$ swaps   
$u_{1,1}$ and $v_{0,1}$. Then $\sigma \in K$, it maps each of $X_1$ and 
$X_2$ to itself, and it swaps $X_3$ and $X_4$. Figure~\ref{fig2} shows us that 
$\sigma(u_{l+1,1})=v_{3l,1}$. 
On the other hand, as $X_1$ is preserved by $\sigma$, $\sigma(v_{-1,0})=v_{-1,0}$. 
These give rise to a contradiction because $u_{l+1,1} \sim v_{-1,0}$ but  
$v_{3l,1} \not\sim v_{-1,0}$. 
\medskip

\begin{figure}[t]
\begin{center}
\begin{tikzpicture}[scale=1.25] ---
\fill (0,2) circle (1.5pt); 
\fill (0,3) circle (1.5pt);
\fill (1,1) circle (1.5pt);
\fill (1,2) circle (1.5pt); 
\fill (1,3) circle (1.5pt); 
\fill (2,0) circle (1.5pt); 
\fill (2,1) circle (1.5pt); 
\fill (2,2) circle (1.5pt); 
\fill (2,3) circle (1.5pt); 
\fill (3,0) circle (1.5pt);
\fill (3,1) circle (1.5pt);
\fill (3,2) circle (1.5pt);
\fill (3,3) circle (1.5pt);
\draw (0,2) node[above]{\footnotesize $u_{-1,1}$};
\draw (0,3) node[above]{\footnotesize $v_{3l-2,1}$};
\draw (1,1) node[above]{\footnotesize $v_{3l-1,0}$};
\draw (1,2) node[above]{\footnotesize $u_{0,0}$};
\draw (1,3) node[above]{\footnotesize $v_{-1,0}$};
\draw (2,0) node[above]{\footnotesize $u_{l+1,1}$};
\draw (2,1) node[above]{\footnotesize $v_{0,1}$};
\draw (2,2) node[above]{\footnotesize $u_{1,1}$};
\draw (2,3) node[above]{\footnotesize $v_{3l,1}$};
\draw (3,0) node[above]{\footnotesize $u_{l+2,0}$};
\draw (3,1) node[above]{\footnotesize $v_{l+1,0}$};
\draw (3,2) node[above]{\footnotesize $u_{2,0}$};
\draw (3,3) node[above]{\footnotesize $v_{1,0}$};
\draw (0.5,2) node[below]{\footnotesize $X_1$};
\draw (1.5,1) node[below]{\footnotesize $X_2$};
\draw (2.5,2) node[below]{\footnotesize $X_3$};
\draw (2.5,0) node[below]{\footnotesize $X_4$};
\draw (0,3) -- (3,0) -- (2,0) -- (3,1) -- (1,1) -- (3,3) -- (2,3) -- (3,2) -- (0,2) -- (1,3) -- (0,3);
\draw (1,3) -- (2,0);
\draw[gray] (-0.3,-0.3) -- (0.3,-0.3) -- (0.3,3.5) -- (-0.3,3.5) -- (-0.3,-0.3); 
\draw (0,-0.5) node[below]{\footnotesize $\Delta_{-1,1}$};
\draw[gray, ] (0.7,-0.3) -- (1.3,-0.3) -- (1.3,3.5) -- (0.7,3.5) -- (0.7,-0.3); 
\draw (1,-0.5) node[below]{\footnotesize $\Delta_{0,0}$};
\draw[gray] (1.7,-0.3) -- (2.3,-0.3) -- (2.3,3.5) -- (1.7,3.5) -- (1.7,-0.3); 
\draw (2,-0.5) node[below]{\footnotesize $\Delta_{1,1}$};
\draw[gray] (2.7,-0.3) -- (3.3,-0.3) -- (3.3,3.5) -- (2.7,3.5) -- (2.7,-0.3); 
\draw (3,-0.5) node[below]{\footnotesize $\Delta_{2,0}$};
\end{tikzpicture}
\caption{The 4-cycles $X_1, X_2, X_3$ and $X_4$ in the graph 
$\cdc(R_{4l}(l+2,l+1))$.} 
\label{fig2}
\end{center}
\end{figure}
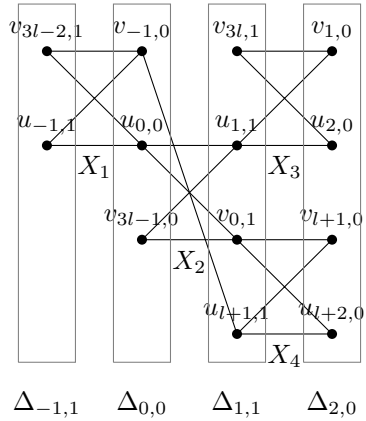

\noindent{\bf Case~2.}~$(a,r)=(l+2,l-1)$ or $(l-2,l+1)$. 
\medskip

We discuss only the case when $(a,r)=(l+2,l-1)$ because the 
other case can be handled in the same way. For $i \in \Z_l$ and $j \in \Z_2$,  
$$
\cdc(\G)[\Delta_{i,j} \cup \Delta_{i+1,j+1}] \cong 2\C_8.
$$ 

The  subgraph of $\cdc(\G)$ induced by the set 
$\Delta_{0,0} \cup\cdots\cup \Delta_{12,0}$ is shown in Figure~\ref{fig3}.  
Label the vertices in the blocks $\Delta_{i,j}$ with numbers from $1$ to $8$ as 
shown in the figure.  For $\sigma \in K$ and for an integer $k \ge 0$, 
let $\sigma_k$ be the permutation of $\{1,\ldots,8\}$ induced by the restriction of $\sigma$ to $\Delta_{i,j}$, where 
$k \equiv i\!\!\pmod l$ and $k \equiv j\!\!\pmod 2$. Note that $\sigma_i=\sigma_j$ 
holds if $i \equiv j\!\! \pmod {2l}$. 
The $8$-cycles intersect at either no vertex or a
pair of antipodal vertices. Moreover, each $8$-cycle between $\Delta_{i,j}$ and 
$\Delta_{i+1,j+1}$ has every other vertex in $\Delta_{i,j}$. These imply that 
$\sigma_{i+1}$ is determined by 
$\sigma_i$. 

Choose $\sigma \in G_{u_{0,0}} \cap G_{u_{1,1}}$. Clearly, $\sigma \in K$ and  
$$
\sigma_0(i)=\sigma_1(i)=i~\text{if}~1 \le i \le 4.
$$
Denote by $X_1$ and $X_2$ the two $8$-cycles between $\Delta_{1,1}$ and 
$\Delta_{2,0}$ such that $u_{1,1} \in \vertex(X_1)$. Then $\sigma$ maps 
each of $X_1$ and $X_2$ to itself, and it is easy to see that it acts on these cycles 
either as the identity or a specific reflection. Denote by $X_3$ the $8$-cycle between 
$\Delta_{0,0}$ and $\Delta_{1,1}$ not through $u_{0,0}$. It follows that 
a vertex of $X_3$ is either fixed by $\sigma$ or it is switched with the vertex 
antipodal to it in $X_3$. This implies that $\sigma$ acts on $X_3$ either as the identity or the reflection with respect to the center of $X_3$. Therefore,  
$\sigma_0$ is the identity permutation or $\sigma_0=(5,7)(6,8)$.  

Assume for the moment that $\sigma_0=(5,7)(6,8)$. 
Then 
$$
\sigma_0=\sigma_1~\text{and}~\sigma_2=(1,6)(2,5)(3,8)(4,7).
$$
A tedious, although straightforward computation, yields that \\ [+1ex]
$\sigma_3=(1,5)(2,8)(3,7)(4,6),~\sigma_4=(1,5)(2,8)(3,7)(4,6),~
\sigma_5=(1,8)(2,7)(3,6)(4,5)$, \\ [+1ex]
$\sigma_6=\sigma_7=(1,3)(2,4),~
\sigma_8=(1,8)(2,7)(3,6)(4,5),~\sigma_9=(1,7)(2,6)(3,5)(4,8)$,\\ [+1ex]
$\sigma_{10}=(1,7)(2,6)(3,5)(4,8),~\sigma_{11}=\sigma_2,~\sigma_{12}=\sigma_0$. 
\medskip

\noindent
This shows that $\sigma_i=\sigma_j$ if and only if $i \equiv j\!\!\pmod {12}$.  
On the other hand, $\sigma_0=\sigma_{2l}$, contradicting our assumption that 
$l$ is odd. We conclude that $\sigma_0$ must be the identity, implying that 
$G_{u_{0,0}} \cap G_{u_{1,1}}=1$, whence $|G_{u_{0,0}}|=4$. 

It can be shown in the same way as above that the existence of a non-identity automorphism in $K_{u_{0,0}}$ forces   
that $3$ divides $l$. But then $\G=R_{12s}(3s+2,3s-1)$, where $l=3s$, and therefore, 
it is edge-transitive due to Proposition~\ref{KKM1}. This and the fact that 
$|G_{u_{0,0}}|=4$ show that $\G$ is stable, a contradiction. 
\end{proof}

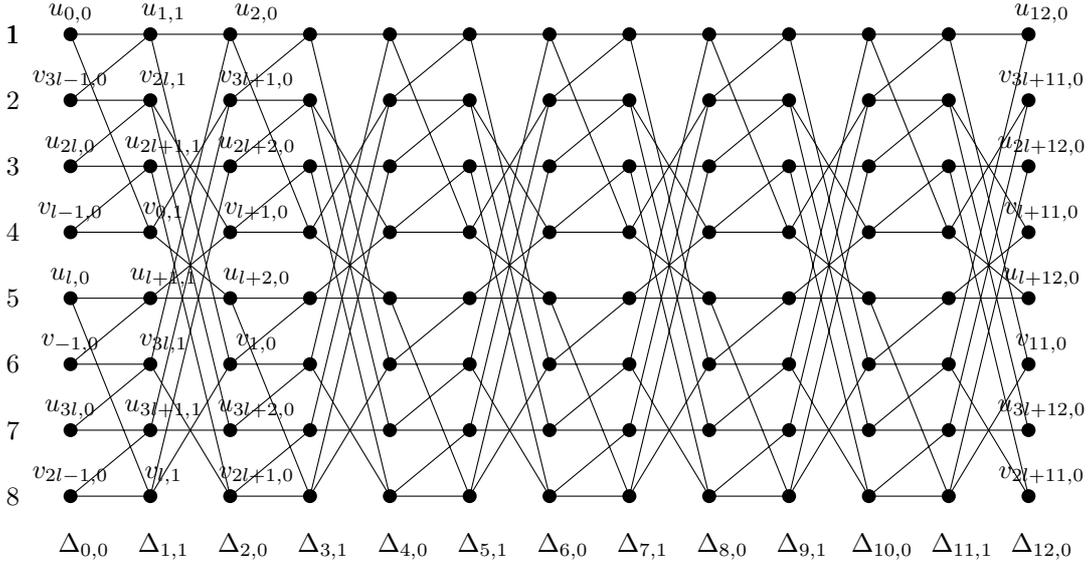
\begin{figure}[t]
\begin{center}
\begin{tikzpicture}[scale=1.75] ---
\foreach \x in {0,0.5,1,1.5,2,2.5,3,3.5}
\draw (-0.3,3.5) node[left]{\footnotesize $1$};
\draw (-0.3,3) node[left]{\footnotesize $2$};
\draw (-0.3,2.5) node[left]{\footnotesize $3$};
\draw (-0.3,2) node[left]{\footnotesize $4$};
\draw (-0.3,1.5) node[left]{\footnotesize $5$};
\draw (-0.3,1) node[left]{\footnotesize $6$};
\draw (-0.3,0.5) node[left]{\footnotesize $7$};
\draw (-0.3,0) node[left]{\footnotesize $8$};
\foreach \x in {0,0.6,1.2,1.8,2.4,3,3.6,4.2,4.8,5.4,6,6.6,7.2}
\foreach \y in {0,0.5,1,1.5,2,2.5,3,3.5}
\fill (\x,\y) circle (1.5pt); 
\draw (0,3.5) node[above]{\footnotesize $u_{0,0}$};
\draw (0,3) node[above]{\footnotesize $v_{3l-1,0}$};
\draw (0,2.5) node[above]{\footnotesize $u_{2l,0}$};
\draw (0,2) node[above]{\footnotesize $v_{l-1,0}$};
\draw (0,1.5) node[above]{\footnotesize $u_{l,0}$};
\draw (0,1) node[above]{\footnotesize $v_{-1,0}$};
\draw (0,0.5) node[above]{\footnotesize $u_{3l,0}$};
\draw (0,0) node[above]{\footnotesize $v_{2l-1,0}$};
\draw (0.7,3.5) node[above]{\footnotesize $u_{1,1}$};
\draw (0.7,3) node[above]{\footnotesize $v_{2l,1}$};
\draw (0.7,2.5) node[above]{\footnotesize $u_{2l+1,1}$};
\draw (0.7,2) node[above]{\footnotesize $v_{0,1}$};
\draw (0.7,1.5) node[above]{\footnotesize $u_{l+1,1}$};
\draw (0.7,1) node[above]{\footnotesize $v_{3l,1}$};
\draw (0.7,0.5) node[above]{\footnotesize $u_{3l+1,1}$};
\draw (0.7,0) node[above]{\footnotesize $v_{l,1}$};
\draw (1.4,3.5) node[above]{\footnotesize $u_{2,0}$};
\draw (1.4,3) node[above]{\footnotesize $v_{3l+1,0}$};
\draw (1.4,2.5) node[above]{\footnotesize $u_{2l+2,0}$};
\draw (1.4,2) node[above]{\footnotesize $v_{l+1,0}$};
\draw (1.4,1.5) node[above]{\footnotesize $u_{l+2,0}$};
\draw (1.4,1) node[above]{\footnotesize $v_{1,0}$};
\draw (1.4,0.5) node[above]{\footnotesize $u_{3l+2,0}$};
\draw (1.4,0) node[above]{\footnotesize $v_{2l+1,0}$};
\draw (7.3,3.5) node[above]{\footnotesize $u_{12,0}$};
\draw (7.3,3) node[above]{\footnotesize $v_{3l+11,0}$};
\draw (7.3,2.5) node[above]{\footnotesize $u_{2l+12,0}$};
\draw (7.3,2) node[above]{\footnotesize $v_{l+11,0}$};
\draw (7.3,1.5) node[above]{\footnotesize $u_{l+12,0}$};
\draw (7.3,1) node[above]{\footnotesize $v_{11,0}$};
\draw (7.3,0.5) node[above]{\footnotesize $u_{3l+12,0}$};
\draw (7.3,0) node[above]{\footnotesize $v_{2l+11,0}$};
\foreach \x in {0,1.2,2.4,3.6,4.8,6}
\draw (\x,0) -- (\x+0.6,0.5) -- (\x,0.5) -- (\x+0.6,1) -- (\x,1) -- (\x+0.6,1.5) -- (\x,1.5) -- (\x+0.6,0) -- (\x,0); 
\foreach \x in {0,1.2,2.4,3.6,4.8,6}
\draw (\x,2) -- (\x+0.6,2.5) -- (\x,2.5) -- (\x+0.6,3) -- (\x,3) -- (\x+0.6,3.5) -- (\x,3.5) -- (\x+0.6,2) -- (\x,2); 
\draw (0.6,3.5) -- (1.2,3.5) -- (0.6,1) -- (1.2,0) -- (0.6,2.5) -- (1.2,2.5) 
-- (0.6,0) -- (1.2,1) -- (0.6,3.5);
\draw (0.6,3) -- (1.2,2) -- (0.6,1.5) -- (1.2,1.5) -- (0.6,2) -- (1.2,3) 
-- (0.6,0.5) -- (1.2,0.5) -- (0.6,3);
\foreach \x in {1.8,3,4.2,5.4,6.6} 
\draw (\x,3.5) -- (\x+0.6,3.5) -- (\x,1) -- (\x+0.6,0) -- (\x,2.5) -- (\x+0.6,2.5) 
-- (\x,0) -- (\x+0.6,1) -- (\x,3.5);
\foreach \x in {1.8,3,4.2,5.4,6.6} 
\draw (\x,3) -- (\x+0.6,2) -- (\x,1.5) -- (\x+0.6,1.5) -- (\x,2) -- (\x+0.6,3) 
-- (\x,0.5) -- (\x+0.6,0.5) -- (\x,3);
\draw (0.1,-0.2) node[below]{\footnotesize $\Delta_{0,0}$};
\draw (0.7,-0.2) node[below]{\footnotesize $\Delta_{1,1}$};
\draw (1.3,-0.2) node[below]{\footnotesize $\Delta_{2,0}$};
\draw (1.9,-0.2) node[below]{\footnotesize $\Delta_{3,1}$};
\draw (2.5,-0.2) node[below]{\footnotesize $\Delta_{4,0}$};
\draw (3.1,-0.2) node[below]{\footnotesize $\Delta_{5,1}$};
\draw (3.7,-0.2) node[below]{\footnotesize $\Delta_{6,0}$};
\draw (4.3,-0.2) node[below]{\footnotesize $\Delta_{7,1}$};
\draw (4.9,-0.2) node[below]{\footnotesize $\Delta_{8,0}$};
\draw (5.5,-0.2) node[below]{\footnotesize $\Delta_{9,1}$};
\draw (6.1,-0.2) node[below]{\footnotesize $\Delta_{10,0}$};
\draw (6.7,-0.2) node[below]{\footnotesize $\Delta_{11,1}$};
\draw (7.3,-0.2) node[below]{\footnotesize $\Delta_{12,0}$};
\end{tikzpicture}
\caption{The subgraph of the canonical double cover $\cdc(R_{4l}(l+2,l-1))$ induced by the set 
$\Delta_{0,0} \cup\cdots\cup \Delta_{12,0}$. }
\label{fig3}
\end{center}
\end{figure}

\begin{lem}\label{L7} 
Assuming Hypothesis~\ref{hypo2}, suppose that $m$ is even. Then $\G/L$ is unstable.
\end{lem}
\begin{proof}
Let $M=\sg{\rho^m,\beta}$, where $\beta$ is the automorphism of $\cdc(\G)$ 
defined in \eqref{eq:beta}. 
On the contrary, assume that $\G/L$ is stable. It follows from this and 
Lemma~\ref{cor-HM} that $M \lhd G$. By Lemma~\ref{L5},  
$\rho^m \in Z(G)$, implying that 
$C_G(M)=C_G(\beta)$. Consider the action of $G$ on itself by conjugaction. By the orbit-stabilizer lemma, 
$|G|=|C_G(\beta)|\cdot |\beta^G|$, where $\beta^G$ denotes the conjugacy class 
of $G$ containing $\beta$. Using also that $\beta \not \in Z(G)$ and $M \lhd G$, we find that $\beta^G=\{\beta,\beta\rho^m\}$, and therefore,  
$C_G(M)=C_G(\beta)$ has index $2$ in $G$. 
Note that the group $C_G(\beta)$ is the group of expected automorphisms of 
$\cdc(\G)$ due to Lemma~\ref{cor-HM}, in particular, 
\begin{equation}\label{eq:centralizer}
C_G(M) \cong \aut(\G) \times \Z_2.
\end{equation}

Since $\cdc(R_m(a,r))$ is edge-transitive and $R_m(a,r)$ is stable, it follows 
that $R_m(a,r)$ is edge-transitive. Let $\Sigma=R_m(a',r')$, where 
$a'=\min(a,m-a)$ and $r'=\min(r,m-r)$. Then $\Sigma$ is one of the graphs in 
the families (a)--(d) in Proposition~\ref{KKM1}.  If $R_m(a,r)$ is bipartite, then 
so is $\G=R_n(a,r)$, see Lemma~\ref{bipartite}, which is not the case. 
Thus $R_m(a,r)$ is not bipartite, which yields that it is 
neither in the family (a) nor in (d). Lemma~\ref{L6} implies that it cannot be in the family  
(b). Thus $\Sigma$ is in the family (c), i.e., $m=12l$ and  
$(a',r')=(3l+2,3l-1)$ or $(3l-2,3l+1)$. Then 
\begin{equation}\label{eq:(a,r)}
(a,r) \in \big\{\, (3l+2,3l-1),\,  (3l-2,3l+1),\, (9l-2,9l+1),\, (9l+2,9l-1)\, \big\}.
\end{equation}
Note that $l$ must be odd because $\G=R_n(a,r)$ is not bipartite. 
One can check that $r^2 \not\equiv 1\!\!\pmod n$. Combining this with Propositions~\ref{DKM1}--\ref{KKM1} yields 
$|\aut(\G)|=2n$. Using also \eqref{eq:centralizer}, we find  
$C_G(M)=\sg{\rho,\mu,\beta}$, where $\mu$ is the automorphism of $\cdc(\G)$ 
defined in \eqref{eq:mu}. 

Choose $\sigma \in G_{u_{0,0}}$ such that $\sigma(u_{1,1})=v_{0,1}$. 
Then $\sigma \notin C_G(M)$.  
Recall that $G^+$ is the subgroup of $G$ that preserves 
the biparts of $\cdc(\G)$. Then $G^+ \cap C_G(M)=\sg{\rho,\mu}$.  
Since $\sigma \in G^+$, it follows that $\sigma$ normalizes $\sg{\rho,\mu}$, and 
therefore, it normalizes also $C=\sg{\rho}$. The proof of Lemma~\ref{L4} shows 
that $a$ is even. This is in contradiction with \eqref{eq:(a,r)} and the fact that $l$ is odd.  
\end{proof}

We are ready to handle the case when $C$ is not normal in $G$. 

\begin{lem}\label{L8} 
Assuming Hypothesis~\ref{hypo2}, suppose that $C \ntriangleleft G$. Then 
$m$ is odd and $(a,r)=(m-2,m-1)$ or $(2,m-1)$.
\end{lem}

Note that if $(a,r)=(m-2,m-1)$, then $\G$ is in the family W2, and if 
$(a,r)=(2,m-1)$, then it is in the family W4.

\begin{proof}
We proceed by induction on $n$. 
With the help of the computer algebra package 
{\sc Magma}~\cite{BCP} we checked that the lemma holds if $n \le 12$. 
From now on it is assumed that $n > 12$.

\begin{claim}
$m$ is odd
\end{claim}
\begin{proof}[Proof of the claim] 
On the contrary, assume that $m$ is even, say $m=2l$. By Lemma~\ref{L7}, $R_m(a,r)$ is  unstable. The graph $R_m(a,r)$ is not bipartite, for otherwise,    
so is $\G=R_n(a,r)$ due to Lemma~\ref{bipartite}, which is 
contrary to Hypothesis~\ref{hypo2}. 

Assume for the moment that $R_m(a,r)$ has two vertices with the same 
neighborhood. Consequently, $\cdc(R_m(a,r))$ has two vertices having 
identical neighborhood too. As $\cdc(R_m(a,r))$ is vertex-transitive, see the remark after Hypothesis~\ref{hypo2}, there exists a vertex of $\cdc(\G)$, which is different from $u_{0,0}$ and has the same neighborhood as $u_{0,0}$.  
On the other hand, a direct check shows that $u_{i,j}$ and $u_{0,0}$ have different neighborhood if $(i,j) \ne (0,0)$; whereas $v_{i,j}$ and $u_{0,0}$ share the same neighborhood if and 
only if $i \equiv \pm 1\!\!\pmod m$, $j=0$,   
$a \equiv \pm 2\!\!\pmod m$ and $r \equiv\pm 1\!\!\pmod m$.  
But then $a$ is even and $r$ is odd, contradicting the fact that $R_m(a,r)$ is not bipartite.
To sum up, $R_m(a,r)$ is non-trivially unstable. 

Note that $C/L$ is not normal in $G/L$, and therefore, 
the induction hypothesis can be applied to $R_m(a,r)$. As a result, $l$ is odd and 
$$
(\min(a,m-a),\min(r,m-r))=(l-2,l-1)~\text{or}~(2,l-1).
$$
Combining this with Lemma~\ref{L5} yields  
$$
(a,r) \in \big\{ (2,l-1), (2,l+1), (m-2,l-1), (m-2,l+1) \big\}.
$$

We consider only the case when $(a,r)=(2,l-1)$, the other pairs can be handled 
in the same way. 
The first lines in the proof of Lemma~\ref{L6} can be applied again 
and this results in the block system $\B$ for $G$ with blocks 
$\Delta_{i,j}$ defined in \eqref{eq:Delta} with $d=l$ (now $l >3$ because $n > 12$). 
The quotient graph $\cdc(\G)/\B$ is a cycle of length $2l$. Let $K$ be the kernel of the action of $G$ on $\B$. For $i \in \Z_l, j \in \Z_2$, 
$$
\cdc(\G)[\Delta_{i,j} \cup \Delta_{i+1,j+1}] \cong \C_{16}.
$$

Consider part of the $16$-cycles $X_1, X_2$ and $X_3$ shown in Figure~\ref{fig4}. 
Choose $\sigma \in G_{u_{0,0}}$ such that $\sigma$ swaps    
$u_{1,1}$ and $v_{0,1}$. Then $\sigma$ maps each of $X_1, X_2$ and $X_3$ 
to itself. Considering the action of $\sigma$ on $X_1$, we find 
$\sigma(u_{2,0})=u_{2,0}$. Then considering its action on $X_2$ yields 
$\sigma(v_{1,0})=v_{3l+1,0}$. Finally, considering 
its action on $X_3$ yields $\sigma(u_{3,1})=v_{2,1}$. This gives rise to a contradiction
because $v_{1,0} \sim u_{3,1}$ and $v_{3l+1,0} \not\sim v_{2,1}$.

\begin{figure}[t]
\begin{center}
\begin{tikzpicture}[scale=1.25] ---
\fill (0,2) circle (1.5pt); 
\fill (1,1) circle (1.5pt);
\fill (1,2) circle (1.5pt); 
\fill (2,0) circle (1.5pt); 
\fill (2,1) circle (1.5pt); 
\fill (2,2) circle (1.5pt); 
\fill (3,1) circle (1.5pt);
\fill (3,0) circle (1.5pt);
\draw (0,2) node[above]{\footnotesize $u_{0,0}$};
\draw (1,1) node[above]{\footnotesize $v_{0,1}$};
\draw (1,2) node[above]{\footnotesize $u_{1,1}$};
\draw (2,2) node[above]{\footnotesize $v_{1,0}$};
\draw (2,1) node[above]{\footnotesize $u_{2,0}$};
\draw (2,0) node[above]{\footnotesize $v_{3l+1,0}$};
\draw (3,1) node[above]{\footnotesize $u_{3,1}$};
\draw (3,0) node[above]{\footnotesize $v_{2,1}$};
\draw (0.5,2.5) node[below]{\footnotesize $X_1$};
\draw (1.5,1) node[below]{\footnotesize $X_2$};
\draw (2.5,1) node[below]{\footnotesize $X_3$};
\draw (2,0) -- (0,2) -- (2,2) -- (3,1) -- (1,1) -- (2,1) -- (1,2) (3,0) -- (2,1);
\draw[gray] (-0.3,-0.3) -- (0.3,-0.3) -- (0.3,3) -- (-0.3,3) -- (-0.3,-0.3); 
\draw (0,-0.5) node[below]{\footnotesize $\Delta_{0,0}$};
\draw[gray, ] (0.7,-0.3) -- (1.3,-0.3) -- (1.3,3) -- (0.7,3) -- (0.7,-0.3); 
\draw (1,-0.5) node[below]{\footnotesize $\Delta_{1,1}$};
\draw[gray] (1.7,-0.3) -- (2.3,-0.3) -- (2.3,3) -- (1.7,3) -- (1.7,-0.3); 
\draw (2,-0.5) node[below]{\footnotesize $\Delta_{2,0}$};
\draw[gray] (2.7,-0.3) -- (3.3,-0.3) -- (3.3,3) -- (2.7,3) -- (2.7,-0.3); 
\draw (3,-0.5) node[below]{\footnotesize $\Delta_{3,1}$};
\end{tikzpicture}
\caption{Part of the 16-cycles $X_1, X_2$ and $X_3$ in the graph 
$\cdc(R_{4l}(2,l-1))$.} 
\label{fig4}
\end{center}
\end{figure}
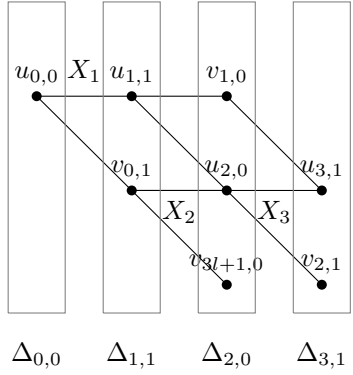
\end{proof}

Since $m$ is odd, it follows that 
$$
\cdc(\G/L) \cong \cdc(R_m(a,r)) \cong R_{2m}(a',r'),
$$
where $a'$ and $r'$ are defined in \eqref{eq:a'r'}.
Let $\Sigma=R_{2m}(a'',r'')$ where $a''=\min(a',2m-a')$ and $r''=\min(r',2m-r')$.
By Proposition~\ref{iso}, $\Sigma$ is edge-transitive, hence it is one of the graphs in the families (a)--(d) of Theorem~\ref{KKM1}.

The fact that $\Sigma$ is bipartite shows that $\Sigma$ cannot be in the 
family (b). As $m$ is odd, it is clear that it cannot be in the family (c) neither. 
If $\Sigma$ is in the family $(d)$, then it follows that $C \lhd G$; this is impossible. 
We conclude that $\Sigma$ is in the family (a), i.e., $a' \equiv \pm 2\!\!\pmod n$ and 
$r' \equiv \pm 1\!\!\pmod n$. If $a$ is even, then $a'=a \le m$, hence $a=2$; and if 
$a$ is odd, then $a'=a+m > m$, hence $a=m-2$. Similarly, we find that 
$r=1$ or $m-1$. As $(a,r) \ne (2,1)$, because $R_n(a,r)$ is not bipartite,  
we have that 
$$
(a,r) \in \big\{ (2,m-1), (m-2,1), (m-2,m-1) \big\}.
$$

We finish the proof by excluding the possibility $(a,r)=(m-2,1)$. 
On the contrary, assume that $(a,r)=(m-2,1)$. Then $G$ admits the block system 
$\B$ with blocks $\Delta_{i,j}$ defined in \eqref{eq:Delta} with $l=m$. 
Let $K$ be the kernel of the action of $G$ on $\B$.
Label the vertices in $\Delta_{i,j}$ with numbers from $1$ to $4$ as 
shown  in Figure~\ref{fig5}.  For $\sigma \in K$ and for an integer $k \ge 0$, 
let $\sigma_k$ be the permutation of $\{1,2,3,4\}$ induced by the restriction of 
$\sigma$ to $\Delta_{i,j}$, where $k \equiv i\!\! \pmod m$ and $k \equiv j\!\! \pmod 2$. 
Note that $\sigma_i=\sigma_j$ holds if $i \equiv j\!\! \pmod {2m}$. 
Choose $\sigma \in G_{u_{0,0}}$ such that $\sigma$ swaps $u_{1,1}$ and 
$v_{m+2,1}$. Then we compute  
$$
\sigma_0=(2,4),~\sigma_1=(1,2)(3,4),~
\sigma_2=(1,3),~\sigma_3=(1,4)(2,3),~\sigma_4=\sigma_0. 
$$
This shows that $\sigma_i=\sigma_j$ if and only if $i \equiv j\!\!\pmod 4$.  
On the other hand, $\sigma_0=\sigma_{2m}$, contradicting the fact that $m$ is odd. 

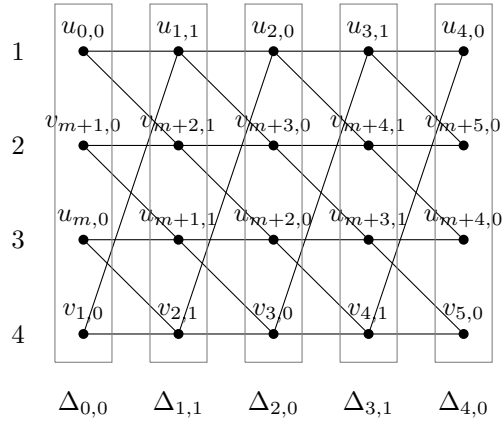
\begin{figure}[t]
\begin{center}
\begin{tikzpicture}[scale=1.25] ---
\fill (0,3) circle (1.5pt); 
\fill (0,2) circle (1.5pt);
\fill (0,1) circle (1.5pt);
\fill (0,0) circle (1.5pt);
\fill (1,3) circle (1.5pt);
\fill (1,2) circle (1.5pt); 
\fill (1,1) circle (1.5pt); 
\fill (1,0) circle (1.5pt); 
\fill (2,3) circle (1.5pt); 
\fill (2,2) circle (1.5pt); 
\fill (2,1) circle (1.5pt); 
]\fill (2,0) circle (1.5pt); 
\fill (3,3) circle (1.5pt);
\fill (3,2) circle (1.5pt);
\fill (3,1) circle (1.5pt);
\fill (3,0) circle (1.5pt);
\fill (4,3) circle (1.5pt);
\fill (4,2) circle (1.5pt);
\fill (4,1) circle (1.5pt);
\fill (4,0) circle (1.5pt);
\draw (0,3) node[above]{\footnotesize $u_{0,0}$};
\draw (0,2) node[above]{\footnotesize $v_{m+1,0}$};
\draw (0,1) node[above]{\footnotesize $u_{m,0}$};
\draw (0,0) node[above]{\footnotesize $v_{1,0}$};
\draw (1,3) node[above]{\footnotesize $u_{1,1}$};
\draw (1,2) node[above]{\footnotesize $v_{m+2,1}$};
\draw (1,1) node[above]{\footnotesize $u_{m+1,1}$};
\draw (1,0) node[above]{\footnotesize $v_{2,1}$};
\draw (2,3) node[above]{\footnotesize $u_{2,0}$};
\draw (2,2) node[above]{\footnotesize $v_{m+3,0}$};
\draw (2,1) node[above]{\footnotesize $u_{m+2,0}$};
\draw (2,0) node[above]{\footnotesize $v_{3,0}$};
\draw (3,3) node[above]{\footnotesize $u_{3,1}$};
\draw (3,2) node[above]{\footnotesize $v_{m+4,1}$};
\draw (3,1) node[above]{\footnotesize $u_{m+3,1}$};
\draw (3,0) node[above]{\footnotesize $v_{4,1}$};
\draw (4,3) node[above]{\footnotesize $u_{4,0}$};
\draw (4,2) node[above]{\footnotesize $v_{m+5,0}$};
\draw (4,1) node[above]{\footnotesize $u_{m+4,0}$};
\draw (4,0) node[above]{\footnotesize $v_{5,0}$};
\draw (-0.5,3) node[left]{\footnotesize $1$};
\draw (-0.5,2) node[left]{\footnotesize $2$};
\draw (-0.5,1) node[left]{\footnotesize $3$};
\draw (-0.5,0) node[left]{\footnotesize $4$};
\draw (0,3) -- (4,3) (0,2) -- (4,2)  (0,1)-- (4,1) (0,0) -- (4,0); 
\draw (0,3) -- (3,0) (0,2) -- (2,0)  (0,1) -- (1,0) (1,3) -- (4,0) (2,3) -- (4,1) (3,3) -- (4,2);
\draw (0,0) -- (1,3) (1,0) -- (2,3) (2,0) -- (3,3) (3,0) -- (4,3);
\draw[gray] (-0.3,-0.3) -- (0.3,-0.3) -- (0.3,3.5) -- (-0.3,3.5) -- (-0.3,-0.3); 
\draw (0,-0.5) node[below]{\footnotesize $\Delta_{0,0}$};
\draw[gray, ] (0.7,-0.3) -- (1.3,-0.3) -- (1.3,3.5) -- (0.7,3.5) -- (0.7,-0.3); 
\draw (1,-0.5) node[below]{\footnotesize $\Delta_{1,1}$};
\draw[gray] (1.7,-0.3) -- (2.3,-0.3) -- (2.3,3.5) -- (1.7,3.5) -- (1.7,-0.3); 
\draw (2,-0.5) node[below]{\footnotesize $\Delta_{2,0}$};
\draw[gray] (2.7,-0.3) -- (3.3,-0.3) -- (3.3,3.5) -- (2.7,3.5) -- (2.7,-0.3); 
\draw (3,-0.5) node[below]{\footnotesize $\Delta_{3,1}$};
\draw[gray] (3.7,-0.3) -- (4.3,-0.3) -- (4.3,3.5) -- (3.7,3.5) -- (3.7,-0.3); 
\draw (4,-0.5) node[below]{\footnotesize $\Delta_{4,0}$};
\end{tikzpicture}
\caption{The subgraph of $\cdc(R_{2m}(m-2,1))$ induced by the set 
$\Delta_{0,0} \cup \cdots \cup \Delta_{4,0}$.} 
\label{fig5}
\end{center}
\end{figure}
\end{proof}

\section*{Acknowledgement}
The authors are grateful to the anonymous referees for their valuable comments which helped 
to improve the text considerably.

\end{document}